\documentclass[11pt,reqno]{amsart} 
\usepackage[utf8]{inputenc} 
\usepackage[left=2.65cm,right=2.65cm,top=2.7cm,bottom=2.7cm]{geometry}
\usepackage{graphicx} 
\usepackage{color}
\usepackage{mathtools}
\usepackage{xcolor}

\usepackage{array} 
\usepackage{paralist} 
\usepackage{verbatim} 
\usepackage{subfig} 

\usepackage{amsfonts}
\usepackage{amsmath}
\usepackage{amsthm}
\usepackage{amssymb}
\usepackage{hyperref}
\usepackage{dsfont}

\hypersetup{
	colorlinks=true,
	linkcolor=blue, 
	citecolor=blue,
	filecolor=blue,
	urlcolor=blue,
}

\makeatletter
\renewcommand\subsubsection{\@secnumfont}{\bfseries}%
\renewcommand\subsubsection{\@startsection{subsubsection}{3}
  \z@{.5\linespacing\@plus.7\linespacing}{-.5em}%
  {\normalfont\bfseries}}
  
  \makeatother

\newcommand{\mel}{\MoveEqLeft}

\newtheorem{theorem}{Theorem}[section]
\newtheorem{proposition*}{Proposition\textsuperscript{*}}
\newtheorem{corollary}[theorem]{Corollary}
\newtheorem{corollary*}{Corollary\textsuperscript{*}}
\newtheorem{proposition}[theorem]{Proposition}
\newtheorem{lemma}[theorem]{Lemma}
\newtheorem{conjecture}[theorem]{Conjecture}
\theoremstyle{definition}

\newtheorem{remark}[theorem]{Remark} 

\newtheorem{example*}{Example\textsuperscript{*}}

\numberwithin{equation}{section}

\newtheorem{manualtheoreminner}{Theorem}
\newenvironment{manualtheorem}[1]{%
  \IfBlankTF{#1}
    {\renewcommand{\themanualtheoreminner}{\unskip}}
    {\renewcommand\themanualtheoreminner{#1}}%
  \manualtheoreminner
}{\endmanualtheoreminner}

\usepackage{bbm}

\def\Limes#1#2 {\lim\limits_{#1\rightarrow #2}}

\DeclareMathOperator{\sgn}{sgn}

\def\eps{\epsilon}

\def\R{\mathbb{R}}

\def\tst{T}

\newcommand{\Chi}{\mathcal{X}}

\def\XXint#1#2#3{{\setbox0=\hbox{$#1{#2#3}{\int}$ }
\vcenter{\hbox{$#2#3$ }}\kern-.59\wd0}}

\renewcommand{\epsilon}{\varepsilon}

\def\scalar#1#2{\langle #1,#2 \rangle}
\def\de{\partial}

\def\dx{\,\mathrm{d}x}

\def\dz{\,\mathrm{d}z}
\def\dy{\,\mathrm{d}y}

\newcommand{\red}[1]{{\textcolor{red}{#1}}}

\newcommand{\E}{\mathcal{E}}

\newcommand{\david}[1]{\textcolor{red}{#1}}

\newcommand\minus\backslash

\newcommand\lan\langle
\newcommand\ran\rangle

\newcommand{\supp}{\operatorname{Supp}}

\newcommand{\dd}{{\,\mathrm d}}

\DeclareMathOperator\dist{dist} 
 
\DeclareMathOperator\diam{diam}

\newcommand{\norm}[1]{\left\lVert#1\right\rVert}

\renewcommand\leq\leqslant
\renewcommand\geq\geqslant

\usepackage{hyperref}

\title[Confinement of multiple vortices]{Long time confinement of multiple concentrated vortices}

 \author[D. Meyer]{David Meyer}
 \address{ \vspace{-0.4cm}
\newline 
\textbf{{\small David Meyer}} 
\vspace{0.15cm}
\newline \indent Instituto de Ciencias Matem\'aticas, Consejo Superior de Investigaciones Cient\'\i ficas, 28049 Madrid, \indent Spain}
 \email{david.meyer@icmat.es}

\keywords{}

\subjclass[2020]{}
\begin{document}

\begin{abstract}
We study the stability of multiple almost circular concentrated vortices in a fluid evolving according to the two-dimensional Euler equations. We show that, for general configurations, they must remain concentrated on time-scales much longer than previously known as long as they remain separated. We further prove a new stability estimate for the logarithmic interaction energy as part of the proof.
\end{abstract}

\maketitle

\section{Introduction}
We are concerned with the behaviour of the two-dimensional Euler equations in vorticity form, which read as \begin{align}
&\de_t\omega+(u\cdot\nabla)\omega=0\quad \text{ in $\Omega$}\label{Euler1}\\
&u=\nabla^\perp (-\Delta)^{-1}\omega\quad \mkern12mu\text{ in $\Omega$}\label{Euler2}\,,
\end{align}
here $\Omega$ is either $\R^2$ or a bounded simply connected subset of $\R^2$ with a sufficiently regular boundary, and in case $\Omega$ is bounded, the inverse Laplacian is equipped with Dirichlet boundary conditions.
By a classical theorem of Yudovich \cite{Yudovich}, these equations are globally well-posed if the initial datum $\omega^0$ is in $L^\infty(\Omega)\cap L^1(\Omega)$.

A classical topic is the study of the stability of specific solutions to \eqref{Euler1}-\eqref{Euler2}.
One such class of solutions, which is observed to be quite stable, are circular vortices, i.e.\ solutions where $\omega$ is (in a loose sense) concentrated on some more or less circular subdomains.
A typical setup for this is to consider initial data $\omega^0$ such that \begin{align}
&\omega^0=\sum_{i=1}^n \omega_i^0\label{A1}\tag{A1}\\
&\supp \omega_i^0\subset B_{\eps N_1}(Y_i^0)\quad \text{for each $i=1,\dots, n$}\label{A2}\tag{A2}\\
&\int_\Omega \omega_i^0\dx=a_i\qquad\qquad\, \text{for each $i=1,\dots, n$},\label{A3}\tag{A3}
\end{align}
where $\eps<<1$, the number $N_1$ is fixed and the points $Y_i^0$ have pairwise distances $>>\eps$.

If $G(x,y)$ denotes the Green's function of the Dirichlet-Laplacian, then one expects that in the limit $\eps\searrow 0$, we have $\omega_i\overset{\ast}{\rightharpoonup} a_i \delta_{Y_i}$ and the velocity in the limit should (at least formally) be \begin{align*}
u^p(x):=\sum_{i=1}^n -a_i\nabla^\perp G(x,Y_i).
\end{align*}
This velocity, however, diverges at $x=Y_i$, more precisely if $\gamma$ is the reflection term from the boundary in the Green's function, i.e.\ \begin{align*}
&\Delta_y \gamma(x,y)=0\quad \text{ for $y\in\Omega$}\\
&\gamma(x,y)=\frac{1}{2\pi}\log|x-y|\quad \text{ for $y\in\de \Omega$}
\end{align*}
(interpreted as $0$ on the full space), then it holds that \begin{align}
G(x,y)=-\frac{1}{2\pi}\log|x-y|+\gamma(x,y)\label{ref green}
\end{align}
and \begin{align*}
u^p(x)=\sum_i\frac{(x-Y_i)^\perp}{2\pi|x-Y_i|^2}-\nabla^\perp \gamma(x,Y_i).
\end{align*}
Here, only the first summand is singular at $Y_i$ (presuming the $Y_i$ are not at the boundary), furthermore, its rotational symmetry indicates that this part of the velocity does not affect $Y_i$ itself and therefore the $Y_i$ should evolve according to the equations \begin{align}
Y_i'=u_i^p(Y_i):=-a_i\nabla^\perp \gamma(Y_i,Y_i)-\sum_{j\neq i}-a_j\nabla^\perp G(Y_i,Y_j).\label{pvs}
\end{align}
This is the so-called point vortex system, which was first introduced by Helmholtz in the 19th century \cite{Helmholtz}. Unlike the actual 2D Euler equations, solutions to this ODE system can form singularities in the form of collisions of the vortices with each other or the boundary in finite time (see e.g.\ \cite[Thm.\ 2.2.3]{Newton_2010} for an example).

In spite of the simple nature of this model, it is in general not trivial at all to show that the vortices actually remain concentrated enough to justify such an approximation.

The first mathematically rigorous justification of this system was by Marchioro and Pulvirenti in 1983 \cite{marchioro1983euler}, with many further improvements e.g.\ in \cite{turkington1987evolution,marchioro1993vortices, caprini2015concentrated,butta2018long, ceci2021vortex,guo2025stability},
where it is shown that the vortices remain confined for a timescale of $O(|\log\eps|)$ and converge to the solution of the point vortex system presuming they do not collide with each other. An alternative approach based on gluing techniques, giving more information on the vortices at the price of stronger initial assumptions, was established in \cite{davila2020gluing}.

For a single vortex in the full space, much better results are available, the best result being confinement to a small ball on a timescale of $O(\eps^{-2}|\log\eps|^{-1})$ from Gamblin, Iftimie and Sideris in \cite{iftimie1999evolution}, furthermore, global in time nonlinear stability results are available for perturbations of circular vortices, see e.g.\ \cite{marchioro1985some,wan1985nonlinear,sideris2009stability,choi2022stability}, though typically in a rather weak sense (e.g.\ in the $L^1$-norm).
For some special configurations of multiple vortices (e.g.\ expanding configurations or near stationary points), results on higher time-scales (or even global) are also known, see e.g.\ \cite{butta2018long,donati2021long,davila2023global,choi2024stability}.
There are also many works constructing stationary, rigidly rotating, quasiperiodic or periodic solutions close to such solutions of the point vortex system, see e.g.\  \cite{burbea1982motions,turkington1985corotating,smets2010desingularization,berti2023time,hassainia2023rigorous,hassainia2024desingularization}

For solutions of Navier-Stokes with sufficiently small viscosity, a modified version of the point vortex approximation has been justified e.g.\ in \cite{gallay2011interaction,marchioro1990vanishing,ceci2024dynamics,dolce2024long}.
Similar questions for SQG, the axisymmetric or helical 3D Euler equations, the lake equations or vortex sheets were studied e.g.\ in \cite{geldhauser2020point,benedetto2000motion,davila2024leapfrogging,hientzsch2024dynamics,glass2018point,donati2024dynamics,enciso2025desingularization}. Let us also mention that from the viewpoint of numerics, understanding how \eqref{pvs} converges to the Euler equation in the limit $n\rightarrow \infty$ can also be interesting, see e.g.\  \cite{goodman1990convergence,rosenzweig2022mean}.

Some open question about the behaviour of vortices are for instance (nonlinear) inviscid damping (see e.g.\ the partial results in \cite{bedrossian2019vortex,ionescu2022axi}), how such vortices behave under collisions, the behaviour of vortex filaments in the 3D Euler equations \cite{jerrard2017vortex} or the justification of higher order models (see e.g.\ \cite[Chapter 6]{Newton_2010}). Negative results illustrating the limitations of these convergence or confinement statements also do not exist to the best of the author's knowledge.

In particular, it is also conjectured (\cite[p. 3]{butta2018long}) that these convergence results actually hold on much longer timescales than $O(|\log\eps|)$ due to difficult-to-capture cancellation effects (which will be explained at the beginning of Section \ref{vortprsec} below).

The purpose of this work is to provide a proof of this conjecture under somewhat stronger assumptions on the initial data, giving confinement on a time-scale of up to $O(\eps^{-1}|\log\eps|^{-\frac{1}{2}})$ and under a slightly stronger form of the no-collision assumption.

The additional assumption we will make is that each initial vortex's magnitude is very
close to being radially decreasing. 
The reason why this assumption is useful is that vortices which are radially symmetric with a radially nonincreasing vorticity are automatically nonlinear stable due to two variational principles: Namely among all measure-preserving rearrangements of the vorticity, they maximize the kinetic energy and minimize the angular momentum, which due to the transport structure of the vorticity equation \eqref{Euler1}  means that the kinetic energy and the momentum automatically both act as a Lyapunov functional for perturbations since they are conserved under the evolution.

Of course, when we are dealing with multiple vortices, the energy of a single vortex is not preserved anymore. We will however see that the change of the energy essentially enjoys a fourth-order estimate and therefore a quantitative version of this variational principle, which we will discuss in detail in Section \ref{sec11} below, will allow for strong control over the vortices.

Before stating our results, we need to introduce some notation. We set \begin{equation}
\E(f):=\frac{-1}{2\pi} \int_{\R^{2+ 2}}\log|x-y|f(x)f(y)\dx\dy,\label{def E}
\end{equation}
which after a formal partial integration equals the kinetic energy $\int_{\R^2}u^2\dx$ for $\omega=f$.
If $\Omega=\R^2$, it can happen that this partial integration fails due to insufficient decay at infinity. Nevertheless, $\mathcal{E}(\omega)$ is still easily seen to be a conserved quantity for $\omega^0\in L^\infty(\R^2)$ with compact support.
Furthermore, for a measurable $f\geq 0$ with suplevel sets of finite measure, the symmetric decreasing rearrangement of $f$ is defined as the unique function $f^*$ with \begin{align}
&|\{f\geq a\}|=|\{f^*\geq a\}| \quad \forall a>0\label{sym rea}\\
&\text{ Every set $\{f^*\geq a\}$ is a ball centered at $0$ for $a>0$.}\label{sym rea2}
\end{align}
See, for instance, \cite[Chapter 3]{LiebLoss} for background reading. In particular, by a classical result of Riesz \cite[Thm.\ 3.7]{LiebLoss}, $f^*$ maximizes $\E$ among all rearrangements of $f$ if $f\geq 0$.

With these definitions at hand, our assumptions on the initial data are the following (in addition to \eqref{A1}-\eqref{A3}):
\begin{align}
&\norm{\omega_i^0}_{L^\infty}\leq N_2\eps^{-2}\label{A4}\tag{A4}\\
&\text{Each $a_i$ is $\neq 0$ and $\omega_i^0\geq 0$ if $a_i> 0$ resp.\ $\omega_i^0\leq 0$ if $a_i<0$}\label{A5}\tag{A5}\\
&\mathcal{E}(|\omega_i^0|^*)-\mathcal{E}(\omega_i^0)\leq N_3 \eps^\beta \text{ for some fixed $\beta> \frac{2}{3}$},\label{A6}\tag{A6}
\end{align}
where $N_2$ and $N_3$ are fixed but arbitrary constants independent of $\eps$ and $\omega_i^0$ is extended to the full space by $0$ in case $\Omega$ is a bounded domain.

Here, the condition \eqref{A6}, which is not used in the aforementioned previous works, measures how close each $\omega_i^0$ is to being radially symmetric. Using that $\log|\cdot|$ is the fundamental solution of the Laplacian, it is not difficult to show that the energy difference in \eqref{A6} is controlled by the $H^{-1}$-difference and that \eqref{A6} is implied by the following condition: \begin{align}
\norm{(\omega_i^0)^*-\omega_i}_{H^{-1}}\leq N_3'\eps^\beta.\tag{A6'}
\end{align}
%
We next define the evolved vortices. For this, we consider the flow map $\Chi:\Omega\times [0,\infty)\rightarrow \Omega$ defined via \begin{align}
\Chi(0,x)=x\quad\text{and}\quad\frac{\mathrm{d}}{\mathrm{d}t}\Chi(t,x)=u(t,\Chi(t,x)).
\end{align}
It is classical that this is well-defined as $u$ is log-Lipschitz, that $\Chi(t,\cdot)$ is a homeomorphism for each fixed $t$ and that $\omega^t=\omega^0\circ\Chi^{-1}(t,\cdot)$, see e.g.\ \cite[Chapter 8]{Majda2001VorticityAI} for background reading.

We then set \begin{align}
\omega_i^t=\omega_i^0\cdot\Chi^{-1}(t,\cdot).
\end{align}
Furthermore we need to define how to assign a point $Y_i$ to these vortices: We take $X_i$ to be the center of mass of $\omega_i$, i.e.\ \begin{align*}
X_i(t):=\frac{1}{a_i}\int_\Omega \omega_i^t x\dx,
\end{align*}
and write $X_i^0$ for the centers of mass of the initial data, and of course \eqref{A1}-\eqref{A3} should hold with the $X_i^0$ in place of the $Y_i$.

Finally, we can only expect a result under the assumption that the vortices do not collide with each other or the boundary. Similarly as in \cite{butta2018long}, we will make the slightly stronger assumption that we do not just forbid collisions but also the scenario that the minimal distances of the vortices go to $0$ as $t\rightarrow +\infty$, i.e.\ we assume:

 \begin{align}
\text{There is some fixed $b>0$ such that} \min_{i,\,j}\Big(\dist(X_i(t),X_j(t)),\dist(X_i(t),\de\Omega)\Big)\geq b.\label{A7}\tag{A7}
\end{align}
We remark that one can replace $b$ on the right-hand side in this assumption with $b\eps^{\alpha}$ for a small fixed $\alpha>0$ at the price of obtaining worse exponents (depending on $\alpha)$ in the theorem below. Let us also stress that, unlike the other assumptions, this is an assumption depending on $t$.

Our main result is then the following. 
\begin{theorem} \label{T1}
Suppose that the assumptions \eqref{A1}-\eqref{A7} hold for the initial data. Then there exists a $C_0=C(\Omega,n,b,N_1,N_2,N_3, \beta,a_1,\dots, a_n)$, not depending on $\eps$, such that there is a time $T$ with \begin{align}
T\geq \begin{cases}
C_0 \eps^{-1}|\log\eps|^{-\frac{1}{2}} \quad&\text{for $\beta>2$}\\
C_0 \eps^{-1}|\log\eps|^{-\frac{2}{3}}\quad&\text{for $\beta=2$}\\
C_0 \eps^{-\frac{\beta}{2}}|\log\eps|^{-\frac{1}{2}} \quad&\text{for $\beta\in(\frac{4}{5},2)$}\\
C_0\eps^{-(3\beta-2)}\quad&\text{for $\beta\in(\frac{2}{3},\frac{4}{5})$},\\
\end{cases}\label{thmbd1}
\end{align}
such that for each $t\in [0,T]$ the assumption \eqref{A7} is violated before the time $t$, \textbf{or} the vortices remain confined in the sense that \begin{align}
\max_{i}\diam\supp\omega_i^t\lesssim \eps^{\min(1,\frac{\beta}{2})}(1+t)^\frac{1}{2}+\eps^{\frac{1}{2}+\frac{\beta}{8}}(1+t)^\frac{1}{4}.\label{thmb2}
\end{align}
Furthermore, regarding the convergence to the point vortex system, we have that, if \eqref{A7} is not violated before the time $t\in [0,T]$, then \begin{align}
\left|\de_t X_i(t)-u_i^p(X_i(t),t)\right|\lesssim \eps^{\min(4,2\beta)}(1+t)+\eps^{2+\frac{\beta}{2}}\quad\text{ for $t\in [0,T]$}.\label{thmb3}
\end{align}
\end{theorem} 

The question of whether this implies convergence to the point vortex system on the timescale $T$ is more difficult. It is certainly true that the convergence of the velocity in \eqref{thmb3} implies convergence on a timescale of $O(|\log\eps|)$ by standard ODE arguments, but without further assumptions on the configuration of the point vortices, it is possible that the error in the ODE system grows exponentially, suggesting that any algebraic bound on the difference of the velocities is not sufficient for convergence on longer time-scales. Using a different way of assigning the vortices a point $X_i$, Donati \cite{donati2024construction} also produced an example where confinement holds, but convergence to the point vortex system does indeed fail on algebraic timescales.

In particular, as a consequence of this issue, it is in general not clear whether the condition \eqref{A7} is implied by a condition on the initial data on the full timescale $T$, however one can easily show with a bootstrap argument that if the analogue of \eqref{A7} holds for the solution to the genuine point vortex system, then \eqref{A7} must also hold for the $X_i$ on a time-scale of order $|\log\eps|$.

For positive intensities $a_i$, we can also ensure that \eqref{A7} holds on the full timescale. 

\begin{manualtheorem}{1.1'}\label{T1'}
Assume that all $a_i$ are positive, that $\Omega=\R^2$, assume \eqref{A1}-\eqref{A6} and that \eqref{A7} holds at $t=0$, then there exists a $C_0=C(n,b,N_1,N_2,N_3, \beta,a_1,\dots, a_n,X_1^0,\dots,X_n^0)$ such that the bounds \eqref{thmbd1}, \eqref{thmb2} and \eqref{thmb3} hold.
\end{manualtheorem}

\subsection{Quantitative stability for the interaction energy}\label{sec11}
Our estimates for vortices almost maximizing the kinetic energy are an extension of previous results by Yan and Yao \cite{yan2022sharp}. Since they might be of independent interest, we state them here as a theorem and provide a brief introduction to the topic. 

The fact that radially symmetric and decreasing vortices maximize the kinetic energy is a special fact of the so-called Riesz rearrangement inequality \cite[Thm.\ 3.7]{LiebLoss}, which states that for all compactly supported nonnegative $f_1,f_2,f_3\in L^\infty(\R^d)$ it holds that \begin{align}
\int_{\R^{2d}}f_1(x)f_2(x-y)f_3(y)\dx\dy\leq \int_{\R^{2d}}f_1^*(x)f_2^*(x-y)f_3^*(y)\dx\dy.\label{Riesz}
\end{align}
In general, if $f_1\neq f_3$, the question when equality holds is quite difficult and there are equality cases where the $f_i$ need not be a translation of $f_i^*$, see e.g.\ \cite{burchard1996cases}, making the question of a more quantitative estimate quite challenging, which was first achieved by Christ in \cite{christ2017sharpened} for the special case of indicator functions. In the case where $f_1=f_3$, equality requires that $f_1=f_1^*(\cdot-y)$ for some $y$. In this case quantitative estimates for $\min_y\norm{f_1-f_1^*(\cdot-y)}_{L^1}$ in terms of the energy were given by \cite{burchard2015geometric,burchard2020stability,fusco2020sharp,frank2021proof}.

For more general functions, Yan and Yao proved a result in \cite{yan2022sharp}, which also covers estimates in the $W_2$-distance (with a dependence on the size of the support of $f_1$).

For our purposes, it will be crucial to obtain estimates that provide more information on how far apart in space $f_1$ and $f_1^*$ can be.


We therefore prove an extension of the result of Yan and Yao, which sharply captures the behaviour of the "far away" part of $f_1$ for the logarithmic energy. This requires the use of Wasserstein distances, the $p$-th Wasserstein distance is defined as \begin{align}
 W_{p}(\mu ,\nu ):=\left(\inf _{\gamma \in \Gamma (\mu ,\nu )}\int _{\R^d\times \R^d}d(x,y)^{p}\mathrm {d} \gamma (x,y)\right)^{\frac {1}{p}},
\end{align}
for measures $\mu,\nu$ on $\R^d$ with the same mass, where $\Gamma (\mu ,\nu )$ denotes the set of all measures on $\R^{2d}$ with marginals $\mu$ and $\nu$. See e.g.\ the textbook \cite{villani2021topics} for background reading. 

\begin{theorem}\label{T2}
There exists a constant $C$, such that if $\rho\in L^\infty(\R^2)$ is compactly supported and nonnegative and such that $\E(\rho^*)-\E(\rho)\leq C\norm{\rho}_{L^1}^2$, then one can split $\rho=\rho^c+\rho^f$ such that $\rho^c$ and $\rho^f$ have disjoint supports and such that \begin{itemize}
\item It holds that \begin{align}\supp \rho^c\subset B_{10\diam\supp \rho^*}(y_0)\label{T21}
\end{align}
where $y_0$ is the center of mass of $\rho^c$, i.e.\ $y_0=\frac{\int \rho^c x\dx}{\int \rho^c\dx}$.
\item It holds that \begin{align}
    \E(\rho^*)-\E(\rho)\geq C\left(\frac{\diam\supp\rho^*}{R_0}\right)R_0^{-2} W_2^2(\rho^c,(\rho^c)^*)\label{T22}
\end{align}
\item \begin{align}
\E(\rho^*)-\E(\rho)\geq C\left(\frac{\diam\supp\rho^*}{R_0}\right)\int_{\R^2} \rho^f(x)\left|\log\frac{x-y_0}{\diam\supp\rho^*}\right|\dx\label{T24}
\end{align}
\end{itemize}
here $R_0:=\norm{\rho}_{L^1}^\frac{1}{2}\norm{\rho}_{L^\infty}^{-\frac{1}{2}}.$
\end{theorem}

The role of these quantities and estimates in the proof of Theorem \ref{T1} will be explained in Subsection \ref{ovsec}.

\begin{remark}\begin{itemize}
\item Using the center of mass of $\rho^c$ instead of the center of mass of $\rho$ is necessary, for instance the example $\rho=\mathds{1}_{B_1(0)\cup B_\eps(\eps^{-10}e_1)}$ (where the difference of energies is $\approx \eps^2|\log\eps|$ by direct calculation) shows that $\rho^c$ need to not be close to the center of mass of $\rho$. Similar examples also show the sharpness of the asymptotic in \eqref{T24}. 
\item Our proof also works for power law kernels and on the $\R^d$ (for $d\geq 2$), if $\E_{\alpha}(\rho):=\int \sgn(\alpha)|x-y|^\alpha\rho(x)\rho(y)\dx\dy$ (with $\alpha\in (-d,2))$, then it holds that \begin{align*}
    \E_\alpha(\rho^*)-\E_\alpha(\rho)\geq C_\alpha\left(\frac{\diam\supp\rho^*}{R_0}\right) R_0^{-2+\alpha} W_2^2(\rho^c,(\rho^c)^*)
\end{align*}
and
\begin{align*}
\E_\alpha(\rho^*)-\E_\alpha(\rho)\geq C_\alpha\left(\frac{\diam\supp\rho^*}{R_0}\right)\int_{\R^d}\rho^f |x-y_0|^{\max(0,\alpha)}\dx,
\end{align*}
where $R_0=\norm{\rho}_{L^1}^\frac{1}{d}\norm{\rho}_{L^\infty}^\frac{1}{d}$.
\end{itemize}
\end{remark}

\subsubsection{Organization of the article and notational conventions}
Theorem \ref{T2} is proven in Section \ref{seprt2}, the Theorems \ref{T1} and \ref{T1'} are proven in Section \ref{vortprsec}, which also contains an outline of the proof strategy.

We write $A\lesssim B$ if there is some constant $C$, possibly depending on $N_1,N_2,N_3,b,\beta,\Omega,a_1,\dots$, but not $\eps,f,\rho,\omega$ such that $A\leq CB$. Sometimes we still write the constants and some of the dependencies to highlight them, in this case, the constant is allowed to change its value from line to line. We denote indicator functions with $\mathds{1}$.

\section{Proof of Theorem \ref{T2}}
\label{seprt2}
Our strategy is to first prove the bound on the ``far away'' parts (Step 1 below) and to then use this to ``compactify'' $\rho$, use the estimate by Yan and Yao from \cite{yan2022sharp} for the compactified version (Step 2) and to finally show that the bounds on the compactified version of $\rho$ imply the bounds on $\rho$ (Step 3). We remark that one can skip most parts of steps 2 and 3 if one allows $y_0$ to be any point instead of the center of mass of $\rho^c$.

We first note that both sides of the inequality are homogeneous in $\rho$, so it is not restrictive to assume that $\norm{\rho}_{L^\infty}=1$. Furthermore, it holds that  \begin{align*}
\mel\mathcal{E}(\rho(c\cdot)^*)-\E(\rho(c\cdot))=\int_{\R^{2+2}} \frac{-1}{2\pi c^2}\log\left(\frac{|x-y|}{c}\right)\left(\rho^*(x)\rho^*(y)-\rho(x)\rho(y)\right)\dx\dy\\
&=\int_{\R^{2+2}} \frac{-1}{2\pi c^2}\log|x-y|\big(\rho^*(x)\rho^*(y)-\rho(x)\rho(y)\big)\dx\dy+\frac{1}{2\pi c^2}\log(c)(\norm{\rho^*}_{L^1}^2-\norm{\rho}_{L^1}^2)\\
&=\frac{1}{c^2}\left(\mathcal{E}(\rho^*)-\E(\rho)\right),
\end{align*}
and the left-hand side is easily seen to scale in the same way with respect to dilations. Therefore, we can also rescale space so that it holds that $\norm{\rho}_{L^1}=1$.

We set \begin{align}
D:=\diam \supp \rho^*\geq \frac{2}{\sqrt{\pi}},\label{bd D}
\end{align}
where the inequality follows from the fact that $1=\norm{\rho^*}_{L^1}\leq \frac{\pi}{4}D^2\norm{\rho^*}_{L^\infty}=\frac{\pi}{4}D^2$.
We need to show the bounds \eqref{T22} and \eqref{T24} with a constant depending on $D$.
For the ease of notation, we set \begin{align*}
    \delta:=\E(\rho^*)-\E(\rho).
\end{align*} 
Let \begin{align*}
D'>>\max(\diam\supp \rho, D)
\end{align*}
be finite.

\textbf{Step 1.} In order to show the bound on $\rho^f$, we rewrite the energy, using Fubini and the fundamental theorem of calculus, as \begin{align*}
\mel\E(f)+\frac{1}{2\pi}\log(D')\norm{f}_{L^1}^2=\int_{\R^2\times \R^2}\frac{-1}{2\pi}\left(\log|x-y|-\log(D')\right)f(x)f(y)\dx\dy\\
&=\int_{\R^2\times \R^2}\left(\int_{|x-y|}^{D'}\frac{1}{2\pi z}\dz\right)f(x)f(y)\dx\dy\\
&=\int_{0}^{D'}\frac{1}{2\pi z}\int_{|x-y|\leq z}f(x)f(y)\dx\dy\dz.
\end{align*}
%
%
Therefore, we can write, using the definition of $D$, \begin{equation}\begin{aligned}
\mel\delta=\int_{0}^{D}\frac{1}{2\pi z}\int_{|x-y|\leq z} \rho^*(x)\rho^*(y)-\rho(x)\rho(y)\dx\dy\dz\\
&+\int_{D}^{D'}\frac{1}{2\pi z}\left(1-\int_{|x-y|\leq z}\rho(x)\rho(y)\dx\dy\right)\dz.\label{ineq far part}
\end{aligned}\end{equation}
From the Riesz rearrangement inequality \eqref{Riesz}, applied to $f_2=\mathds{1}_{B_z(0)}$, we infer that the inner term in both integrals is positive for each $z$.  Since $\int_{|x-y|\leq z}\rho(x)\rho(y)\dx\dy$ is increasing in $z$, we see from considering $z\in [D,D+\frac{1}{100}]$, that \begin{align*}
\int_{|x-y|\leq D+\frac{1}{100}}\rho(x)\rho(y)\dx\dy\geq 1-C(D)\delta.
\end{align*}
By the pigeonhole principle and the assumption that $\delta$ is small, we see that this implies the existence of an $x_0$ such that \begin{align}
\int_{|x_0-y|\leq D+\frac{1}{100}}\rho(y)\dy\geq \frac{1}{2}.\label{bd x0}
\end{align}
On the other hand, we see from $\int \rho(x)\rho(y)\dx\dy=1$ and \eqref{ineq far part} that
\begin{align*}
\int_{D}^{D'}\frac{1}{2\pi z}\int_{|x-y|\geq z}\rho(x)\rho(y)\dx\dy\dz\leq \delta.
\end{align*}
We note that $|x-x_0|\leq {D+\frac1{100}}$ and  $|y-x_0|\geq z+D+\frac{1}{100}$ together imply that $|x-y|\geq z$ and therefore obtain together with \eqref{bd x0} and Fubini that \begin{align*}
\delta\geq \int_{D}^{D'}\frac{1}{2\pi z}\int_{x\in B_{D+\frac{1}{100}}(x_0),\, |y-x_0|\geq z+D+\frac{1}{100}}\rho(x)\rho(y)\dx\dy\dz\geq \int_{D}^{D'}\frac{1}{4\pi z}\int_{|y-x_0|\geq D+\frac{1}{100}+z}\rho(y)\dy\dz.
\end{align*}
We can further rewrite this integral as \begin{align*}
\mel\int_{D}^{D'}\frac{1}{4\pi z}\int_{|y-x_0|\geq D+\frac{1}{100}+z}\rho(y)\dy\dz=\frac{1}{4\pi}\int_{|y-x_0|\geq 2D+\frac{1}{100}}\rho(y)\int_{D}^{|y-x_0|-(D+\frac{1}{100})}\frac{1}{z}\dz\dy\\
&\gtrsim\int_{|y-x_0|\geq \frac{5}{2}D}\rho(y)\log\left(\frac{|y-x_0|}{D}\right)\dy,
\end{align*}
where we have used the fact that $D'$ is much bigger than all the other parameters and the lower bound \eqref{bd D} on $D$ to absorb the $+\frac{1}{100}$.

In sum, we have obtained the bound \begin{align}
\int_{|y-x_0|\geq \frac{5}{2}D}\rho(y)\log\left(\frac{|y-x_0|}{D}\right)\dy\leq C(D)\delta.\label{log bd}
\end{align}

\textbf{Step 2.} We next define a suitable compactly supported rearrangement of $\rho$.
There exists a (measure-preserving) rearrangment of $\rho \mathds{1}_{B_{3D}(x_0)^c}$ which is  supported in $B_{4 D}(x_0)\backslash B_{3D}(x_0)$ because $\mathcal{L}^2(\supp \rho)=\frac{\pi}{4}D^2$ (by the definition \eqref{bd D} of $D$) is much smaller than $|B_{4 D}(x_0)\backslash B_{3D}(x_0)|$. Let $\bar{\rho}$ be such a rearrangment of $\rho \mathds{1}_{B_{3D}(x_0)^c}$.
We now define a rearrangment $\tilde{\rho}$ of $\rho$ as \begin{align*}
\tilde{\rho}:=\begin{cases}
\rho &\quad \text{ on $B_{3D}(x_0)$}\\
\bar{\rho}&\quad \text{ on $B_{3D}(x_0)^c$},
\end{cases}
\end{align*}
which is clearly supported on $B_{4D}(x_0)$.

We also note that it holds that \begin{align}
\norm{\rho\mathds{1}_{B_{\frac{5}{2}D}(x_0)^c}}_{L^1}+\norm{\tilde{\rho}\mathds{1}_{B_{\frac{5}{2}D}(x_0)^c}}_{L^1}\lesssim \delta,\label{bd far part}
\end{align}
by \eqref{log bd}, because the logarithm is always $\geq \frac{1}{2}$ there.
We claim that \begin{equation}
\E(\tilde{\rho})\geq \E(\rho)-C(D)\delta.\label{claim}
\end{equation}
for some numerical constant $C(D)$, depending only on $D$, but not on $\delta$ or $\rho$.

Using that $\rho$ and $\tilde{\rho}$ agree in a ball of radius $B_{3D}(x_0)$, we may expand the definition and write \begin{align*}\mel\E(\tilde{\rho})- \E(\rho)=
\int_{\R^2\times \R^2} \frac{-1}{2\pi}\log|x-y| \left(\tilde{\rho}(x)+\mathds{1}_{B_{3D}(x_0)}(x)\tilde{\rho}(x)\right)\bar{\rho}(y)\dx\dy
\\
&-2\int_{B_{\frac{5}{2}D}(x_0)\times B_{3D}(x_0)^c}\frac{-1}{2\pi}\log|x-y|\rho(x)\rho(y)\dx\dy\\
&-\int_{\R^2\times B_{3D}(x_0)^c}\frac{-1}{2\pi}\log|x-y|\Big(\rho(x)\mathds{1}_{ B_{\frac{5}{2}D}(x_0)^c}(x)+\rho(x)\mathds{1}_{B_{3D}(x_0)\backslash B_{\frac{5}{2}D}(x_0)}(x)\Big)\rho(y)\dx\dy\\
&=:I-II-III,
\end{align*}
where $I,II,III$ stand for the terms in each line.

In $I$, we may estimate the logarithm from above with $\log(8D)$ and hence see that \begin{align*}
I\geq -\frac{1}{\pi}\int_{\R^2\times \R^2}\log(8D)\tilde{\rho}(x)\bar{\rho}(y)\dx\dy \geq- \frac{1}{\pi}\log(8D)\norm{\tilde{\rho}}_{L^1}\norm{\bar{\rho}}_{L^1}\geq -C(D)\delta,
\end{align*}
where we used the bound $\norm{\bar{\rho}}_{L^1}\leq \delta$ which follows directly from \eqref{bd far part}.

In $II$, we simply note that the logarithm here is bounded from below by the bound \eqref{bd D} on $D$ and hence by \eqref{bd far part} \begin{equation*}
II\leq \left|\log\frac{1}{2}D\right| \norm{\rho}_{L^1}\norm{\rho\mathds{1}_{B_{3D}(x_0)^c}}_{L^1}\leq  C(D)\delta.
\end{equation*}
In $III$, we need to only consider the points with $|x-y|\leq 1$ to get a lower bound. Using that $\mathds{1}_{B_1(0)}\log|\cdot|$ is in every $L^p$ for $p<\infty$ and Young's convolution inequality, we hence see that \begin{align*}
III&\geq -\norm{\rho\mathds{1}_{ B_{\frac{5}{2}D}(x_0)^c}+\rho\mathds{1}_{B_{3D}(x_0)\backslash B_{\frac{5}{2}D}(x_0)}}_{L^1}\norm{\rho\mathds{1}_{ B_{3D}(x_0)^c}}_{L^2}\norm{\mathds{1}_{B_1(0)}\log|\cdot|}_{L^2}\\
&\gtrsim- \norm{\rho\mathds{1}_{ B_{\frac{5}{2}D}(x_0)^c}+\rho\mathds{1}_{B_{3D}(x_0)\backslash B_{\frac{5}{2}D}(x_0)}}_{L^1}\norm{\rho\mathds{1}_{ B_{3D}(x_0)^c}}_{L^1}^\frac{1}{2}\norm{\rho\mathds{1}_{ B_{3D}(x_0)^c}}_{L^\infty}^\frac{1}{2}\\
&\gtrsim -\delta^\frac{3}{2},
\end{align*}
where we have used the bound \eqref{bd far part} on the $L^1$-norm.
Together, the claim \eqref{claim} follows.

\textbf{Step 3.}
We now have \begin{align*}
\mathcal{E}(\rho^*)-\mathcal{E}(\tilde{\rho})\leq C(D)\delta.
\end{align*}
Furthermore $\supp \tilde{\rho}\subset B_{4D}(x_0)$ by definition and therefore we may apply the $W^2$-stability estimate of Yan and Yao \cite[Thm.\ 1.2]{yan2022sharp} to see that \begin{equation}
W_2^2(\rho^*(\cdot-x_1),\tilde{\rho})\leq C(D) \delta,\label{tilde bd}
\end{equation}
where $x_1$ is the center of mass of $\tilde{\rho}$ (i.e.\ $x_1=\int \tilde{\rho}(x) x\dx$) and $\tilde{\rho}^*=\rho^*$ by definition.

We set \begin{align*}
    &\rho^c:=\mathds{1}_{3D}(x_0)\tilde{\rho}\\
    &y_0:=\frac{1}{\int_{B_{3D}(x_0)}\rho(x)\dd x}\int\rho^c(x) x\dx.
\end{align*}
We then have that \begin{align}
|y_0-x_1|\leq C(D)\delta.\label{cmass est}
\end{align}
Indeed, it follows from the mass estimate \eqref{bd far part} that \begin{align*}
|y_0-x_1|\lesssim \left|1-\frac{1}{\int_{B_{3D}(x_0)}\rho(x)\dd x}\right|\int_{B_{3D}(x_0)} \rho(x)|x-x_1|\dx+ \int_{\R^2\backslash B_{3D}(x_0)}\tilde{\rho}(x)|x-x_1|\dd x\lesssim C(D)\delta.
\end{align*}
%
%
%
Since furthermore $x_1$ must lie in the support of $\tilde{\rho}$ and therefore in $B_{4D}(x_0)$, we hence see \eqref{T21} if $\delta<<1$. The estimate \eqref{T24} follows from \eqref{log bd}. To see \eqref{T22}, we make use of the metric property of the Wasserstein distance and the fact that it is controlled by the $L^1$-norm for compactly supported functions (\cite[Prop.\ 7.10]{villani2021topics}) to see that \begin{align*}
\mel W_2^2((\rho^c)^*(\cdot-y_0),\rho^c)\\
&\leq 4\Big(W_2^2\big((\rho^c)^*(\cdot-x_1),(\rho^c)^*(\cdot-y_0)\big)+W_2^2\big((\rho^c)^*(\cdot-x_1),\norm{\rho^c}_{L^1}\tilde{\rho}^*(\cdot-x_1)\big)\\
&\quad+W_2^2\big(\norm{\rho^c}_{L^1}\tilde{\rho}^*(\cdot-x_1),\norm{\rho^c}_{L^1}\tilde{\rho}\big)+W_2^2\big(\rho^c,\norm{\rho^c}_{L^1}\tilde{\rho}\big)\Big)\\
&\leq C(D)\left( |y_0-x_1|^2+\delta+\norm{\tilde{\rho}-\rho^c}_{L^1}\right)\\
&\leq C(D) \delta,
\end{align*}
where we have used \eqref{cmass est}, \eqref{tilde bd} and \eqref{bd far part} again.\hfill\qedsymbol

\section{Proof of Thm.\ \ref{T1} and \ref{T1'}}\label{vortprsec}

\subsection{Notation and outline of the main ideas}\label{ovsec}


We denote the velocities in the point vortex system by $u^p$ and $u_i^p$ as in the introduction, i.e.\ \begin{align}
u^p(x)&=\sum_{i=1}^n -a_i\nabla^\perp G(x,X_i)\label{def up}\\
u_i^p(x)&=u^p(x)-a_i\frac{(x-X_i)^\perp}{2\pi|x-X_i|^2}.\label{def uip}
\end{align}
%
%
%
The second main idea of the proof (aside from the quantitative energy estimates) is that, as already noticed in \cite{butta2018long}, there should be some massive cancellations in the interactions between the vortices due to the separation of time scales and, at least heuristically, all the interactions from the boundary and the other vortices with $\omega_i$ should be "averaged" over the rotations of $\omega_i$ around itself.
This kind of mechanism is called "averaging principle" in the dynamical systems literature, and we refer to e.g.\ the textbook \cite{sanders1985theory} for further reading.

Our method for implementing this for our problem is to use functionals which would capture this effect for the genuine point vortex system and to use the energy estimates to compare the velocity with the one of the point vortex system (cf.\ Lemma \ref{lem mean value} below).

It is natural to look at the Hamiltonian of the motion of $x$, relative to the motion of $X_i$, which we renormalize to behave like $|x-X_i|$ at leading order; we set



 \begin{align*}
\Psi_i(x):=&a_i\left(G(x,X_i)-\gamma(X_i,X_i)-\nabla\gamma(X_i,X_i)\cdot(x-X_i)\right)\\
&+\sum_{j\neq i} a_j\left(G(x,X_j)-G(X_i,X_j)-\nabla G(X_i,X_j)\cdot(x-X_i)\right)
\end{align*}
and \begin{align*}
d_i(x):=\exp\left(-\frac{2\pi}{a_i}\Psi_i(x)\right)
\end{align*}
(interpreted as $0$ at $x=X_i$), which is a function of $x,X_1,\dots X_n$, written as a function of $x$ only for simplicity.

We observe that $\Psi_i$ is the streamfunction of the velocity $u^p(x)-u_i^p(X_i)$ (modulo a constant), directly from the definitions. Therefore, we must have \begin{equation} \label{d ort}
\nabla_x f(d_i(x))\perp u^p(x)-u_i^p(X_i)
\end{equation}
for every $f\in C^1$.
On the other hand, this function is also almost the distance $|x-X_i|$ as the following Lemma shows, whose proof we postpone.

\begin{lemma}\label{d is nice}
Let $|x-X_i|\leq \min(\frac{1}{2},\frac{b}{2})$, then, whenever \eqref{A7} holds, we have the following estimates \begin{align}
|d_i(x)-|x-X_i||\lesssim |x-X_i|^3\label{d est1}
\end{align}
and \begin{align}
|\mathrm{D}_{X_j}d_i(x)|\lesssim& |x-X_i|^3 \quad \text{ for $j\neq i$}\label{d est2}\\
|\mathrm{D}_{X_i}d_i(x)+\mathrm{D}_x d_i(x)|\lesssim& |x-X_i|^3\label{d est3}\\
\left|\scalar{\mathrm{D}_xd_i(x)}{(x-X_i)^\perp}\right|\lesssim& |x-X_i|^4\label{ort deri est}\\
\left|\mathrm{D}_x d_i(x)\right|\lesssim& 1 \quad\text{ for $x\neq X_i$.}\label{d deri 1}
\end{align}
Furthermore, \begin{align}
\left|\mathrm{D}_x^2 d_i(x)^2-2I\right|\lesssim |x-X_i|^2,\label{2nd deri d}
\end{align}
in particular, this second derivative is uniformly bounded and $d_i^2$ is convex in a neighborhood of $X_i$.
\end{lemma}

It is not difficult to check that with the cancellation \eqref{d ort} and these estimates, one can capture the "averaging principle" for the genuine point vortex velocity. A short calculation which we omit here (and which is also not relevant for the rest of the proof) shows that if $x$ and $X_i$ evolve according to the point vortex system, then it holds that $|\frac{\mathrm{d}}{\mathrm{d}t}d_i(x)|\lesssim |d_i(x)|^3$, which is much better than the naive estimate for $|x-X_i|$.

In our case, we will also have to deal with the errors from the fact that the velocity is not the exact point vortex velocity.
Our strategy for this is to adapt the method used by Gamblin, Iftimie, and Sideris in \cite{iftimie1999evolution} (in the form in which it is presented in the lecture notes \cite{iftimie2004large}) for a single vortex with $d_i$ replacing the distance $|x-X_i|$.
We define the following modified versions of the momenta and the spread, which will be used in the subsequent analysis: \begin{align}
&M_k(t):=\sum_{i=1}^n\int_{\Omega}|\omega_i^t(x)|\eta_\eps(d_i(x))d_i(x)^k\dx\label{def M}\\
&S=\max_{i=1,\dots,n}\max\left(40N_1\eps, \sup_{x\in \supp \omega_i^t}d_i(x)\right).\label{def S}
\end{align}
Here $\eta_\eps=\eta_1(\eps^{-1}\cdot)$ is a smooth non-negative cutoff function, which equals $1$ on $[80N_1\eps,\infty)$ and is supported in $[40N_1\eps,\infty)$ (the constant $N_1$ is the one from \eqref{A2}). 

In order to use Theorem \ref{T2}, we introduce the energy defect as \begin{align*}
\mathcal{D}(t)=\sum_{i=1}^n \E(|\omega_i|^*)-\E(\omega_i^t,
\end{align*}
where $|\omega_i|^*=|\omega_i^t|^*$ is the symmetric decreasing rearrangement of $|\omega_i|$ as defined in \eqref{sym rea}, \eqref{sym rea2}. $|\omega_i|^*$ does not depend on $t$ because the vorticity is transported.
As $|\omega_i|^*$ maximizes $\E$ among all rearrangements of $\omega_i$ if $\omega_i\geq 0$ (resp.\ among all rearrangements of $-\omega_i$ if $\omega_i\leq 0$), and $\E$ is even, this is non-negative, and by the assumption \eqref{A6} it also holds that \begin{align}
 0\leq \mathcal{D}(0)\lesssim \eps^\beta.\label{D init}
\end{align} 
For technical reasons, we will prove all subsequent estimates under the assumption that \begin{align}\label{control cmass}
\frac{M_1}{\eps}+\mathcal{D}+\max_{i,\, x\in \supp \omega_i} |X_i-x| &\leq C_1,\tag{B1}
\end{align}
where $C_1$ is a sufficiently small constant of order $1$, which is allowed to depend on $b, n ,\Omega, N_1$ etc.\ but not on $\eps$. 
We will later see that this assumption holds up to the time $T$ in the theorem.

We will not estimate the derivative of $\mathcal{D}$ itself, instead we will use that \begin{align*}
\mel-\mathcal{D}+\int_{\Omega^2}\sum_{i=1}^n \gamma(x,y)\omega_i(x)\omega_i(y)+\sum_{i\neq j} G(x,y)\omega_i(x)\omega_j(y) \dx\dy\\
&=-\sum_{i=1}^n \mathcal{E}(|\omega_i|^*)+\int_{\Omega^2} G(x,y)\omega(x)\omega(y)\dx\dy
\end{align*}
is a conserved quantity and it is therefore enough to understand the evolution of \begin{align*}\int_{\Omega^2}\sum_{i=1}^n \gamma(x,y)\omega_i(x)\omega_i(y)+\sum_{i\neq j} G(x,y)\omega_i(x)\omega_j(y) \dx\dy,\end{align*} which is almost the energy of the point vortex system.

For technical reasons, it is easier to understand the derivative of the energy of the point vortices at $X_i$. We therefore set \begin{align}
\tilde{\mathcal{D}}:=\sum_{i=1}^n a_i^2\gamma(X_i,X_i)+\sum_{i\neq j} a_ia_j G(X_i,X_j)+\sum_{i=1}^n\E(|\omega_i|^*)-\int_{\Omega^2}G(x,y)\omega(x)\omega(y)\dx\dy.\label{def tilde D}
\end{align}
These two quantities are almost equivalent by the following Lemma.

\begin{lemma}\label{lem d}
Assume \eqref{A1}-\eqref{A7} and \eqref{control cmass}. Then it holds that \begin{align}
\left|\mathcal{D}-\tilde{\mathcal{D}}\right|\lesssim \eps^2\mathcal{D}^\frac{1}{2}+M_2.\label{est dif d}
\end{align}
In particular, it holds that \begin{align}
    \mathcal{D}&\lesssim \eps^4+\tilde{\mathcal{D}}\label{d eq1}\\
    \tilde{\mathcal{D}}&\lesssim \eps^4+\mathcal{D}.\label{d eq2}
\end{align}
\end{lemma}
This Lemma is proven below in Section \ref{sec proof D}.
Its proof crucially relies on Lemma \ref{lem mean value} for the velocities, which in turn is the part that uses Theorem \ref{T2}.

We then have the following differential estimates, whose proof roughly follows the aforementioned scheme by Gamblin, Iftimie, and  Sideris from \cite{iftimie1999evolution}. \begin{lemma}
Under the assumptions \eqref{A1}-\eqref{A7} and \eqref{control cmass}, it holds that \begin{align}
\tilde{\mathcal{D}}(t)'&\lesssim \eps^2\mathcal{D}(t)^\frac{1}{2}+M_2(t)\label{d est d}\\
M_k(t)'&\lesssim kS(t)^2M_k(t)+k\eps^2\mathcal{D}(t)^\frac{1}{2}M_{k-4}(t)+k\mathcal{D}(t)M_{k-2}(t)\notag\\
&\quad+C^k\left(\eps^{k+2}+\eps^{k-2}\mathcal{D}(t)^\frac{1}{2}+\eps^{k-2-\frac{1}{100}}\mathcal{D}(t)+M_1(t)\eps^{k-3}\right)\mathcal{D}(t)\mkern60mu\mathrm{for}\mkern7mu k\geq 4\notag\\
&\quad+k^2\mathcal{D}M_{k-2}(t)|\log\eps|^{-1}+k\eps^{1-\frac{1}{100}}\mathcal{D}(t)M_{k-3}+k^2\mathcal{D}(t)\eps^{2-\frac{4}{100}}M_{k-4}(t)\label{d est m}\\
S(t)'&\lesssim S(t)^3+\eps^2\mathcal{D}(t)^\frac{1}{2}S(t)^{-3}+\mathcal{D}(t)S(t)^{-1}+\eps^{-1} C^kS(t)^{-\frac{k}{2}}M_k(t)^\frac{1}{2}\mkern58mu\mathrm{for}\mkern7mu k>0\label{d est s}
\end{align}
\end{lemma}
These estimates are proven in the Subsections  \ref{sec proof D}, \ref{sec proof m}, and \ref{sec proof S}. We will solve this system of differential inequalities in Section \ref{Gronwallsec}, which will yield the following bounds.

\begin{proposition} \label{eff est}
If the assumptions \eqref{A1}-\eqref{A6} on the initial data are fulfilled, then there exists a time $T$ such that \begin{align}
T\gtrsim \begin{cases}\frac{\eps^{-1}}{|\log\eps|^\frac{1}{2}} \quad&\text{ for $\beta > 2$}\\
\frac{\eps^{-1}}{|\log\eps|^\frac{2}{3}} \quad &\text{ if $\beta=2$}\\
\frac{\eps^{-\frac{\beta}{2}}}{|\log\eps|^\frac{1}{2}}\quad &\text{ for $\beta\in [\frac{4}{5},2)$}\\
\eps^{-(3\beta-2)}\quad &\text{ for $\beta\in (\frac{2}{3},\frac{4}{5})$}
\end{cases}\label{bd t}
\end{align} 
and such that for each $t\in [0,T]$, \eqref{A7} is violated before the time $t$ \textbf{or} for $k\in [4,|\log\eps|]$, it holds that
\begin{align}
\mathcal{D}(t)&\lesssim \eps^4(1+t)^2+\eps^{\beta}\label{bd d}\\
M_k&\lesssim  \begin{cases} C^k\left(\left(\eps^k(1+t)^\frac{k}{2}+\eps^{k(\frac{1}{2}+\frac{\beta}{8})}(1+t)^\frac{k}{4}\right)\left(\eps^4(1+t)^2+\eps^\beta\right)+\eps^{k-2+\frac{3}{2}\beta}(1+t)\right)&\text{ for $\beta\geq 2$}\\
C^k\left(\eps^{\beta(\frac{k}{2}+1)}(1+t)^\frac{k}{2}+\eps^{k(\frac{1}{2}+\frac{\beta}{8})+\beta}(1+t)^\frac{k}{4}+\eps^{k-3+2\beta}(1+t)\right)&\text{ for $\beta\leq 2$}\end{cases}\label{bd m}\\
S(t)&\lesssim \eps^{\min(1,\frac{\beta}{2})} (1+t)^\frac{1}{2}+\eps^{\frac{1}{2}+\frac{\beta}{8}}(1+t)^\frac{1}{4},\label{bd s}
\end{align}
and \eqref{control cmass} is fulfilled for $t'\in [0,t)$.
\end{proposition}

This Proposition immediately implies the confinement bound \eqref{thmb2} in the theorem, by the definition \eqref{def S} of $S$ and because $d_i(x)$ is equivalent to $|x-X_i|$ by \eqref{d est1}.

The convergence of the velocities in \eqref{thmb3} of $X_i$ follows from the bounds on $\mathcal{D}$ and $S$ above and the following Proposition.

\begin{proposition}\label{P5}
Assume \eqref{A1}-\eqref{A7} and \eqref{control cmass}, then we have the following estimates for the difference between the point vortex velocity and $u$: 
 It holds that \begin{align}
\left|u_i^p(X_i)-\frac{\mathrm{d}}{\mathrm{d}t}X_i\right|\lesssim\eps^2\mathcal{D}^\frac{1}{2}+M_2\lesssim \eps^2\mathcal{D}^\frac{1}{2}+S^2\mathcal{D}.\label{bd velo x}
\end{align}
\end{proposition}

We note for future reference that \eqref{bd velo x} together with \eqref{control cmass} and the definition \eqref{def uip} of the $u_i^p$ also implies that \begin{align}
\left|\frac{\mathrm{d}}{\mathrm{d}t}X_i\right|\lesssim 1\quad\text{for all $i$.}\label{dxi}
\end{align}
Theorem \ref{T1'} will follow from the estimates on the energy, it is shown in Subsection \ref{sec38}.

\subsection{ Proof of Lemma \ref{d is nice} and further Preparations}
\begin{proof}[Proof of the Lemma]
Simply expanding the definition and using \eqref{ref green} yields \begin{align}
    d_i(x)=|x-X_i|e^{g(x)}\label{d 1st id}
\end{align}
where \begin{align*}
    g(x):=&-2\pi\big(\gamma(x,X_i)-\gamma(X_i,X_i)-\nabla\gamma(X_i,X_i)\cdot(x-X_i)\big)\\
    &\quad+\frac{-2\pi}{a_i}\sum_{j\neq i} a_j\big(G(x,X_j)-G(X_i,X_j)-\nabla G(X_i,X_j)\cdot(x-X_i)\big)\,.
\end{align*}

\noindent This function $g$ has a double zero at $x=X_i$ and is smooth by the assumption that the $X_j$'s stay away from the boundary and each other (see \eqref{A7}). Therefore, for $|x-X_i|$ sufficiently small \begin{equation*}
\left|e^{g(x)}-1\right|\lesssim |x-X_i|^2,
\end{equation*}
which, together with \eqref{d 1st id} gives \eqref{d est1}.
The same bound also holds for the derivatives with respect to $X_j$ for $j\neq i$ by smoothness, yielding \eqref{d est2}. The bound \eqref{d deri 1} directly follows directly from the fact that everything is smooth as long as the $X_i$ stay away from each other and the boundary.

To see the estimates \eqref{d est3}-\eqref{2nd deri d}, we can rewrite $g$ as a function of the variables $\bar{X}:=x-X_i$, $\tilde{X}=x+X_i$ and the $X_j$ (for $j\neq i$). Then $g$ and $\de_{\bar{X}}g$ both vanish at $\bar{X}=0$ and the same holds for the derivatives of these with respect to the other variables. We can then estimate \begin{align}
|\de_{X_j} d_i(x)|= |x-X_i|e^{g(x)}|\de_{X_j} g(x)|\lesssim |\bar{X}|^3=|x-X_i|^3,
\end{align}
where we have used Taylor's theorem to estimate $|\de_{X_j} g(x)|\lesssim |\bar{X}|^2$. Similarly one obtains \eqref{d est3}, using the fact that $\de_x+\de_{X_i}$ does not act on $|x-X_i|$. \eqref{ort deri est} uses that $\scalar{\de_x |x-X_i|}{(x-X_i)^\perp}=0$. \eqref{2nd deri d} follows straightforwardly from the product and chain rules.
\end{proof}

In order to apply Theorem \ref{T2} well, we will need to relate the point $y_0$ from Theorem \ref{T2}, applied to $\omega_i$, and $X_i$.

For $i$ such that $\omega_i\geq 0$, let $\tilde{X}_i,\, \omega_i^c,\,\omega_i^f$ denote a triple for which the properties \eqref{T21}, \eqref{T22} and \eqref{T24} in Theorem \ref{T2} hold for $y_0=\tilde{X}_i$, $\omega_i^c=\rho^c$, $\omega_i^f=\rho^f$ and $\rho=\omega_i$ (extended by $0$ to function on $\R^2$ if $\Omega\neq \R^2$). For $i$ such that $\omega_i\leq 0$, let $\tilde{X}_i,\, \omega_i^c,\,\omega_i^f$ be such that these properties hold for $\rho=-\omega_i$, $\omega_i^c=-\rho^c$, $\omega_i^f=-\rho^f$ and $y=\tilde{X}_i$.

We first note that it holds that $R_0\approx\diam\supp\omega_i^*\approx \eps$ and hence the constants depending on $\frac{R_0}{\diam\supp\rho^*}$ in Theorem \ref{T2} are bounded independently of $\eps$.


\begin{lemma}
Assume \eqref{control cmass}, then it holds that \begin{align}
|X_i-\tilde{X}_i|\lesssim \eps\mathcal{D} +M_1.\label{dist c mass}
\end{align}
\end{lemma}
\begin{proof}
First note that the assumption \eqref{control cmass} together with the definition of $M_1$ and the fact that $\norm{\omega_i\mathds{1}_{B_{20N_1\eps}(\tilde{X}_i)}}_{L^1}\geq \frac{1}{2}|a_i|$ (by \eqref{T24} and where the $20$ is due to switching from diameter to radius) implies that \begin{equation}
|\tilde{X}_i-X_i|\leq 120N_1\eps,\label{1st bd X} \end{equation}
otherwise we obtain a contradiction from noting that, if \eqref{1st bd X} is false, then $B_{20N_1\eps}(\tilde{X}_i)$ lies in the set where the cutoff function in the definition \eqref{def M} of $M_1$ is $1$ (by \eqref{d est1}) and therefore \begin{align*}
M_1\geq \int_{B_{20N_1\eps}(\tilde{X}_i)}d_i(x)|\omega_i(x)|\dd x\geq \frac{100N_1}{4}\eps|a_i| 
\end{align*}
which is impossible if $C_1$ is chosen sufficiently small.
Therefore we have that $B_{20N_1\eps}(\tilde{X_i})\subset B_{200N_1\eps}(X_i)$ and compute from the definition \begin{align*}
\mel|X_i-\tilde{X}_i|=\left|\frac{1}{|a_i|}\int_\Omega (x-X_i)\omega_i(x)\dd x- \frac{1}{\int_{B_{20N_1\eps}(\tilde{X}_i)}\omega_i(x)\dx}\int_{B_{20N_1\eps }(\tilde{X}_i)}\omega_i(x)(x-X_i)\dd x \right|\\
&\leq 200N_1\eps\int_{B_{200N_1\eps}(X_i)}\left|\frac{1}{a_i}-\frac{1}{\int_{B_{20N_1\eps}(\tilde{X}_i)}\omega_i(x)\dx}\right||\omega_i(x)|\dx\\
&\quad+\left|\frac{1}{a_i}\right|\int_{\Omega\backslash B_{200N_1\eps}(X_i)}|\omega_i(x)||x-X_i|\dx\\
&\lesssim \eps \left|a_i-\int_{B_{20N_1\eps}(\tilde{X}_i)}\omega_i(x)\dx\right|+M_1\lesssim \eps\mathcal{D}+M_1,
\end{align*}
where we have made use of the fact that $|x-X_i|\approx d_i(x)$ on the support of $\omega_i$ by \eqref{d est1} and the assumption \eqref{control cmass} and used Theorem \ref{T2} and the definition \eqref{def M} of $M_1$ in the last line.
\end{proof}

%

Let us collect some easy facts for further reference.
Note that the assumption \eqref{control cmass} and \eqref{dist c mass}  imply $\supp \omega_i^c\subset B_{21N_1\eps}(X_i)$ if $C_1$ is small enough. In particular, only $\omega_i^f$ contributes to $M_k$ and $S$. Furthermore, we may use the energy defect and Theorem \ref{T2} to bound \begin{align}
\int_{\Omega} |\omega_i^f|\log\left|\frac{x-X_i}{\eps}\right|\dx\lesssim \mathcal{D},\label{est log}
\end{align}
in particular \begin{align}
&\int_\Omega |\omega_i^f|\dx\lesssim \mathcal{D}\label{est far part}\\
&\int_\Omega d_i(x)^k|\omega_i^f|\dx\lesssim M_k+C^k\eps^k\mathcal{D}. \label{est momentum}
\end{align} 
for all $i$.

\subsection{Estimates on the velocity and proof of Proposition \texorpdfstring{\ref{P5}}{3.5}}
The Proposition is a special case of part b) of the following Lemma.

\begin{lemma}\label{lem mean value}\begin{itemize}
\item[a)] Let $|x-X_i|\in [25N_1\eps,C_1]$ and assume \eqref{A1}-\eqref{A7} and \eqref{control cmass},  then it holds that \begin{align*}
\left|\int_\Omega \nabla^\perp G(x,y) \omega_i^c(y)\dy-a_i\nabla^\perp G(x,X_i)\right|\lesssim \eps^2\mathcal{D}^\frac{1}{2}|x-X_i|^{-3}+\mathcal{D}|x-X_i|^{-1}+M_1|x-X_i|^{-2}.
\end{align*}

\item[b)] Let $c>0$ be given. Let $F:\Omega\rightarrow \R$ be harmonic in $B_c(X_i)$. Then it holds that  \begin{align*}
    \left|\int_\Omega  F(x)\omega_i(x)\dx-a_i F(X_i)\right|\lesssim_c \norm{F}_{C^2(B_c(X_i))}\left(\eps^2\mathcal{D}^\frac{1}{2}+M_2 \right),
\end{align*}
where the implicit constant does not depend on $F$.
\end{itemize}
\end{lemma}



\begin{proof}[Proof of the Proposition \ref{P5} using Lemma \ref{lem mean value}]

We compute the time derivative of $X_i$ and see by partial integration and expansion of the Biot-Savart law that \begin{align*}
\mel\frac{\mathrm{d}}{\mathrm{d}t}X_i=\frac{1}{a_i}\int_\Omega u(x)\omega_i(x)\dx\\
&=\frac{1}{a_i}\int_\Omega\int_\Omega \left(\frac{(x-y)^\perp}{2\pi|x-y|^2}-\nabla^\perp \gamma(x,y)\right)\omega_i(x)\omega_i(y)\dx\dy\\
&\quad-\frac{1}{a_i}\sum_{j\neq i}\int_\Omega\int_\Omega \nabla^\perp G(x,y)\omega_i(x)\omega_j(y)\dx\dy\\
&=-\frac{1}{a_i}\int_\Omega\int_\Omega \nabla^\perp \gamma(x,y)\omega_i(x)\omega_i(y)\dx\dy-\frac{1}{a_i}\sum_{j\neq i}\int_\Omega\int_\Omega \nabla^\perp G(x,y)\omega_i(x)\omega_j(y)\dx\dy.
\end{align*}
Here, we have used the antisymmetry of the planar Biot-Savart law in the last step.
Observe that because of \eqref{control cmass} and \eqref{A7}, the functions $\omega_i$ and $\omega_j$ have disjoint supports with distances of order $1$. Therefore, we can apply Lemma \ref{lem mean value} b) to the second integral twice with $\nabla^\perp G(\cdot,y)$ resp. $\nabla^\perp G(x,\cdot)$ in place of $F$ and to the first with $\nabla^\perp\gamma$ in place of $F$ and obtain that \begin{align*}
   \mel\frac{\mathrm{d}}{\mathrm{d}t}X_i=-\frac{1}{a_i}\int_\Omega \nabla^\perp\gamma(x,X_i) \omega_i(x)a_i\dx-\frac{1}{a_i}\sum_{j\neq i}\int_\Omega \nabla^\perp G(X_i,y)a_i\omega_j(y)\dy+O\left(\eps^2\mathcal{D}^\frac{1}{2}+M_2\right)\\
   &=-\frac{1}{a_i}a_i^2\nabla^\perp\gamma(X_i,X_i)-\frac{1}{a_i}\sum_{j\neq i}a_ia_j\nabla^\perp G(X_i,X_j)+O\left(\eps^2\mathcal{D}^\frac{1}{2}+M_2\right),
\end{align*}
which yields the statement by the definition \eqref{def uip} of $u_i^p$.
\end{proof}

Before proving the lemma, we also note the following corollary.

\begin{corollary}\label{cor u bd 1}
 Suppose that $|x-X_i|\in [25N_1\eps, C_1]$ and that \eqref{A1}-\eqref{A7} and \eqref{control cmass} hold, then it holds that \begin{equation*}\begin{aligned}
    \mel\left|\int_\Omega \nabla G^\perp(x,y)\left(\omega_i^c(y)+\sum_{j\neq i}\omega_j(y)\right) \dy-u^p(x)\right|\\
    &\lesssim \eps^2\mathcal{D}^\frac{1}{2}|x-X_i|^{-3}+\mathcal{D}|x-X_i|^{-1}+M_1|x-X_i|^{-2}.
\end{aligned}\end{equation*}
\end{corollary}
\begin{proof}
This follows directly from the Lemma by using part a) for $\omega_i^c$ and part b) for the $\omega_j$, where one uses the harmonicity of $\nabla^\perp G$ in a similar way as in the previous proof.
\end{proof}

\begin{proof}[Proof of Lemma \ref{lem mean value}]

It is not restrictive to assume that $a_i\geq 0$, the other case follows by symmetry.

\textbf{a)} First observe that, because of \eqref{dist c mass}, and because of the assumption on $x$ we have that \begin{align}
|x-X_i|\approx |x-\tilde{X}_i|,\label{comp dist}
\end{align}
as we have $|\tilde{X}_i-X_i|< N_1\eps$ if $C_1$ in the assumption \eqref{control cmass} is sufficiently small.

Next, we note that, by the definition of $G$ and because $\gamma$ is smooth away from the boundary, it follows from \eqref{dist c mass} and the assumption on $x$ that \begin{align}
\left|\nabla^\perp G(x,X_i)-\nabla^\perp G(x,\tilde{X}_i)\right|\lesssim |X_i-x|^{-2}|X_i-\tilde{X}_i|\lesssim M_1|x-X_i|^{-2}+\mathcal{D}|x-X_i|^{-1}.\label{vel est1}
\end{align}

\noindent In the next step, we observe that $|\nabla G(x,\tilde{X}_i)|\approx |x-X_i|^{-1}$ by definition and because of \eqref{comp dist}. In particular, by the inequality \eqref{est far part}, we see that \begin{align}
\left|\int_\Omega \omega_i^c\dy\,\nabla^\perp G(x,\tilde{X}_i)-a_i\nabla^\perp G(x,\tilde{X}_i)\right|\lesssim \left|\int_\Omega \omega_i^c\dy-a_i\right||x-X_i|^{-1}\lesssim \mathcal{D}|x-X_i|^{-1}.\label{vel est2}
\end{align}
Let $(\omega_i^c)^*$ denote the symmetric decreasing rearrangement of $\omega_i^c$, as defined in \eqref{sym rea}-\eqref{sym rea2}. This function is supported on a ball of radius $\leq N_1\eps$, since the diameter of its support must be smaller than the one of $\omega_i^0$. In particular, $x$ does not lie in the support of $(\omega_i^c)^*(\cdot-\tilde{X}_i)$ because by assumption $|x-\tilde{X}_i|\geq 25N_1\eps-|\tilde{X}_i-X_i|>N_1\eps$. Therefore, the function $\nabla^\perp G(x,\cdot)$ is harmonic on the support of $(\omega_i^c)^*(\cdot-\tilde{X}_i)$, which, using the mean value principle for harmonic functions, yields that \begin{align}
\int_\Omega \nabla^\perp G(x,y)(\omega_i^c)^*(y-\tilde{X}_i)\dy=\int_\Omega \omega_i^c\dy\,\nabla^\perp G(x,\tilde{X}_i).\label{mvp}
\end{align}
It remains to estimate $\int_\Omega \nabla^\perp G(x,y)\left(\omega_i^c-(\omega_i^c)^*(y-\tilde{X}_i)\right)\dy$. We do this by noting that by the assumption on $x$, we have  \begin{equation*}\dist\left(x,\,\supp\omega_i^c-(\omega_i^c)^*(\cdot-\tilde{X}_i)\right)\gtrsim |X_i-x|\end{equation*} and in particular, $\nabla^\perp G(x,\cdot)$ is $C^2$ on this support.

We may linearize $\nabla^\perp G(x,\cdot)$ around $\tilde{X}_i$ and note that the constant term is $0$ because $\int_\Omega \omega_i^c(y)-(\omega_i^c)^*(y-\tilde{X}_i)\dy=0$. The 
 linear term also drops out because \begin{align*}
\int_\Omega \omega_i^c(y) (y-\tilde{X}_i)\dy=\int_\Omega (\omega_i^c)^* (y-\tilde{X}_i)(y-\tilde{X}_i)\dy=0
\end{align*}
by the definition of $\tilde{X}_i$ as a center of mass of $\omega_i^c$ in Thm.\ \ref{T2}. We therefore see that \begin{align}
    \mel\left|\int_\Omega \nabla^\perp G(x,y)\left(\omega_i^c(y)-(\omega_i^c)^*(y-\tilde{X}_i)\right)\dy\right|\nonumber\\
    &=\left|\int_\Omega \nabla^\perp \left(G(x,y)-G(x,\tilde{X}_i)-\nabla G(x,\tilde{X}_i)\cdot(y-\tilde{X}_i)\right)\left(\omega_i^c(y)-(\omega_i^c)^*(y-\tilde{X}_i)\right)\dy\right|\nonumber\\
    &\lesssim \norm{\nabla^\perp \left(G(x,y)-G(x,\tilde{X}_i)-\nabla G(x,\tilde{X}_i)\cdot(y-\tilde{X}_i)\right)}_{W_y^{1,\infty}\big(\supp\omega_i^c-(\omega_i^c)^*(\cdot-\tilde{X}_i)\big)}\nonumber\\
    &\quad\times W_1\left(\omega_i^c,(\omega_i^c)^*(\cdot-\tilde{X}_i)\right)\nonumber\\
    &\lesssim \eps|x-X_i|^{-3} W_2\left(\omega_i^c,(\omega_i^c)^*(\cdot-\tilde{X}_i)\right)\nonumber\\
    &\lesssim |x-X_i|^{-3}\eps^2\mathcal{D}^{\frac{1}{2}}.\label{wd est}
\end{align}
Here we have made use of the fact that the $W_1$-distance is equivalent to the $W^{-1,1}$-norm and of Hölders inequality for Wasserstein distances (see \cite[Thm.\ 1.14 and (7.3)]{villani2021topics}) and of Theorem \ref{T2}.

The lemma follows from combining \eqref{vel est1}, \eqref{vel est2}, \eqref{mvp} and \eqref{wd est} with the triangle inequality.\smallskip

\textbf{b)} The proof is quite similar to the previous one. We use the original center of mass $X_i$ to linearize this time and observe that \begin{align*}
\int_\Omega  F(x)\omega_i(x)\dx-a_iF(X_i)=\int_\Omega  \big(F(x)-F(X_i)-\nabla F(X_i)\cdot(x-X_i)\big)\omega_i(x)\dx
\end{align*}
where the linear and the constant terms disappear for the same reasons as above.
We now split into the contributions of $\omega_i^c$ and $\omega_i^f$.
As $| F(x)-F(X_i)-\nabla F(X_i)\cdot(x-X_i)|\lesssim |x-X_i|^2$ because $F$ is $C^2$ it holds that \begin{align*}
&\int_\Omega\bigl| F(x)-F(X_i)-\nabla F(X_i)\cdot(x-X_i)\bigr||\omega_i^f(x)|\dx\lesssim_c\norm{F}_{C^2(B_c(X_i))}\int_\Omega |x-X_i|^2|\omega_i^f(x)|\dx\\
&\lesssim_c \norm{F}_{C^2(B_c(X_i))}(M_2+\eps^2\mathcal{D}),
\end{align*}
where we used \eqref{est momentum} and the fact that $d_i(x)$ and $|x-X_i|$ are comparable by Lemma \ref{d is nice}.
For the estimate of the contribution of $\omega_i^c$ we again exploit the mean value principle, using that $F(\cdot)-F(X_i)-\nabla F(X_i)\cdot (\cdot-X_i)$ is harmonic, and observe that \begin{align*}
    \mel\int_\Omega \big(F(x)-F(X_i)-\nabla F(X_i)\cdot(x-X_i)\big)\omega_i^c(x)\dx\\
    &=\int_\Omega \big(
    F(x)-F(X_i)-\nabla F(X_i)\cdot(x-X_i)\big)\left(\omega_i^c(x)-(\omega_i^c)^*(x-\tilde{X}_i)\right)\dx\\
    &\quad+\int_\Omega\omega_i^c\dx\left(F(\tilde{X}_i)-F(X_i)-\nabla F(X_i)\cdot(\tilde{X}_i-X_i)\right).
\end{align*}
We may estimate the second summand, using \eqref{dist c mass} and that $F$ is $C^2$, as \begin{align*}
\mel\norm{\omega_i^c}_{L^1}\left|F(\tilde{X}_i)-F(X_i)-\nabla F(X_i)\cdot(X_i-\tilde{X}_i)\right|\lesssim_c\norm{F}_{C^2(B_c(X_i))} |\tilde{X}_i-X_i|^2\\
&\lesssim_c \norm{F}_{C^2(B_c(X_i))}(M_1^2+\eps^2\mathcal{D}^2),
\end{align*} which has the desired bound because $\mathcal{D}\lesssim 1$ by the assumption \eqref{control cmass}. The first summand on the other hand can be estimated with the Wasserstein distance as in \eqref{wd est}, yielding that \begin{align*}
    \left|\int_\Omega \big(F(x)-F(X_i)-\nabla F(X_i)\cdot(x-X_i)\big)(\omega_i^c(x)-(\omega_i^c)^*(x-\tilde{X}_i))\dx\right|\lesssim \norm{F}_{C^2(B_c(X_i))}\eps^2\mathcal{D}^\frac{1}{2}.
\end{align*}
Putting the estimates together yields the statement after noting that $M_1^2\lesssim M_2$ by Hölder and because $\norm{\omega_i^f}_{L^1}\lesssim \mathcal{D}\lesssim 1$ by \eqref{est far part}.
\end{proof}

\subsection{Proof of Lemma \texorpdfstring{\ref{lem d}}{3.2} and of \texorpdfstring{\eqref{d est d}}{3.17}}
\label{sec proof D}

\begin{proof}[Proof of Lemma \ref{lem d}] By the definition of $\tilde{\mathcal{D}}$, we have \begin{align*}
\left|\mathcal{D}-\tilde{\mathcal{D}}\right|&\leq \sum_{i=1}^n\left|\int_{\Omega^2}\gamma(x,y)\omega_i(x)\omega_i(y)\dx\dy-a_i^2\gamma(X_i,X_i)\right|\\
&\quad+\sum_{i\neq j}\left|\int_{\Omega^2}G(x,y)\omega_i(x)\omega_j(y)\dx\dy-a_ia_jG(X_i,X_j)\right|.
\end{align*}
In the first double integral, we may employ Lemma \ref{lem mean value} b) twice, using that $\gamma$ is harmonic in both variables and smooth away from the boundary, to see that \begin{align*}
\mel\int_{\Omega}\int_\Omega\gamma(x,y)\omega_i(x)\omega_i(y)\dx\dy=a_i\int_\Omega\gamma(X_i,y)\omega_i(y)\dy+O\left(\eps^2\mathcal{D}^\frac{1}{2}+M_2\right)\\
&=a_i^2\gamma(X_i,X_i)+O\left(\eps^2\mathcal{D}^\frac{1}{2}+M_2\right).
\end{align*}
The exact same argument can also be made with the other integrals, since $G$ is also harmonic away from $\{x=y\}$ and the supports of $\omega_i$ and $\omega_j$ have a positive distance by the assumptions \eqref{A7} and \eqref{control cmass}, yielding \begin{align*}
\left|\int_{\Omega^2}G(x,y)\omega_i(x)\omega_j(y)\dx\dy-a_ia_jG(X_i,X_j)\right|\lesssim \eps^2\mathcal{D}^\frac{1}{2}+M_2.
\end{align*}
This shows \eqref{est dif d}. To see \eqref{d eq1} and \eqref{d eq2} on the other hand, we note that \begin{align*}
\eps^2\mathcal{D}^\frac{1}{2}\leq C\eps^4+\frac{1}{C}\mathcal{D}
\end{align*}
for every $C>0$ and $M_2\lesssim C_1 \mathcal{D}$ by \eqref{est far part} and \eqref{control cmass}. Therefore, by choosing $C$ and $C_1$ small and reabsorbing $\mathcal{D}$ into the left hand side, \eqref{est dif d} implies that \begin{align*}
\mathcal{D}\lesssim \eps^4+\tilde{\mathcal{D}},
\end{align*}
which is \eqref{d eq1}

\eqref{d eq2} follows in the same way after using \eqref{d eq1} to estimate $\eps^2\mathcal{D}^\frac{1}{2}\lesssim \eps^4+\eps^2\tilde{\mathcal{D}}^\frac{1}{2}$.
\end{proof}

 We move on to the differential inequality \eqref{d est d}. The idea is that if the velocity were the one of the point vortex system, $\tilde{\mathcal{D}(t)}$ would be a conserved quantity. 
 
The last two terms in the definition \eqref{def tilde D} of $\tilde{\mathcal{D}}$ are conserved under the evolution, since $|\omega_i^t|^*$ does not depend on $t$ due to the transport structure of the Euler equations. We compute, using the symmetry of $G$ and $\gamma$, that \begin{align}\label{deri d}
\tilde{\mathcal{D}}'=&2\sum_{i=1}^n\mathrm{D}\gamma(X_i,X_i)\cdot\frac{\mathrm{d}}{\mathrm{d}t}X_i+\sum_{i\neq j}\mathrm{D}G(X_i,X_j)\cdot\left(\frac{\mathrm{d}}{\mathrm{d}t}X_i+\frac{\mathrm{d}}{\mathrm{d}t}X_j\right).
\end{align}
%
%
%
Observe that by the definition of $u_i^p$, it holds that \begin{align*}
2\sum_{i=1}^n\mathrm{D}\gamma(X_i,X_i)\cdot u_i^p(X_i)+\sum_{i\neq j}\mathrm{D}G(x,y)\cdot(u_i^p(X_i)+u_j^p(X_j))=0.
\end{align*}
Subtracting this from \eqref{deri d} we see that \begin{align*}
|\tilde{\mathcal{D}}'|\lesssim\left(\sup_{i\neq j}|\mathrm{D}G(X_i,X_j)|+|\mathrm{D}\gamma(X_i,X_i)|\right) \sum_{i=1}^n\left|u_i^p-\frac{\mathrm{d}}{\mathrm{d}t}X_i\right|.
\end{align*}
By the assumption \eqref{A7}, the supremum is $\lesssim 1$. The difference on the other hand is bounded by Proposition \ref{P5}, which yields that \begin{align*}
|\tilde{\mathcal{D}}'|\lesssim \eps^2\mathcal{D}^\frac{1}{2}+M_2,
\end{align*}
as desired. \hfill\qedsymbol

\subsection{Proof of the differential estimate \texorpdfstring{\eqref{d est m}}{3.18} for \texorpdfstring{$M_k$}{Mk}}
\label{sec proof m}

%

For the ease of notation we set \begin{equation*}h_i^k=\eta_\eps(d_i(x))d_i(x)^k.
\end{equation*}

We first collect some elementary estimates for $h$.

\begin{lemma}
Assume $|x-X_i|\lesssim \min(\frac{1}{2},\frac{b}{2})$, then it holds that \begin{align}
&0\leq h_i^k(x)\lesssim C^k|x-X_i|^k\label{h est1}\\
&| \mathrm{D}_x h_i^k(x)|\lesssim k h_i^{k-1}(x)+C^k\eps^{k-1}\mathds{1}_{|x-X_i|\lesssim 100N_1\eps}\label{bd deri k}\\
& \frac{\left|\mathrm{D}_xh_i^k(x)-\mathrm{D}_xh_i^k(y)\right|}{|x-y|}\lesssim k^2(h_i^{k-2}(x)+h_i^{k-2}(y))+C^k\eps^{k-2}\mathds{1}_{\min(|x-X_i|,|y-X_i|)\lesssim 100N_1\eps}\quad \mathrm{for}\, k\geq 2.\label{bd 2deri h}
\end{align}
\end{lemma}
\begin{proof}
\eqref{h est1} is immediate from the definition and Lemma \ref{d is nice}.

\eqref{bd deri k} follows from the product rule and the estimate \eqref{d deri 1} as well as the fact that $d_i(x)\lesssim \eps$ on the set where $\nabla \eta_\eps\neq 0$.

For \eqref{bd 2deri h} we use the convexity of $d_i^2$  (Lemma \ref{d is nice}), which implies the convexity of its level sets, which in particular implies also the convexity of the level sets of each $h_i^k$ as they are the same. We also distinguish the cases of whether one of $x$ or $y$ lies in $B_{100N_1\eps}(X_i)$ or not. In the first case, we estimate  \begin{align*}
\frac{\left|\mathrm{D}_xh_i^k(x)-\mathrm{D}_xh_i^k(y)\right|}{|x-y|}\lesssim \sup_{z\in [x,y]} \left|\mathrm{D}_x^2 h_i^k(z)\right|\lesssim\sup_{z\in [x,y]} k^2h_i^{k-2}(z)+C^k\eps^{k-2},
\end{align*}
where $[x,y]$ denotes the line between $x$ and $y$, and the estimate for the second derivative follows directly from the product rule as well as \eqref{2nd deri d}. By the aforementioned convexity of the level sets it holds that $h_i^{k-2}(z)\leq h_i^{k-2}(x)+h_i^{k-2}(y)$.
In the case in which neither $x$ nor $y$ lie in $B_{100N_1\eps}(X_i)$, we use that $h_i^k=d_i^{k}$ on this set and do the same estimate with the second derivative of $d_i^k$, where the summand $C^k\eps^{k-2}$ drops out. \end{proof}


To estimate the derivative of $M_k$, it suffices to estimate the derivative of the contribution of each $i$ separately, furthermore, we may ignore the modulus in the definition because each $\omega_i$ is either nonnegative or nonpositive. 
We now compute, using partial integration \begin{align*}
\mel\de_t \int_\Omega \omega_i(x)h_i^k(x)\dx=\int_\Omega (\de_t h_i^k) \omega_i-h_i^k\nabla\cdot(u\omega_i)\dx=\int_\Omega \left(\sum_{j=1}^n  \mathrm{D}_{X_j}h_i^k \cdot\de_t X_j+\mathrm{D}_x h_i^k\cdot u\right)\omega_i\dx\\
&=\int_\Omega \left(-\mathrm{D}_x h_i^k\cdot \de_t X_i+\mathrm{D}_x h_i^k\cdot u\right)\omega_i^f\dx\\
&\quad+\int_\Omega \left((\mathrm{D}_{X_i}+\mathrm{D}_x)h_i^k\cdot \de_t X_i +\sum_{j\neq i} \mathrm{D}_{X_j}h_i^k \cdot\de_t X_j \right)\omega_i^f \dx,
\end{align*}
where we have used that $\omega_i^c$ is $0$ on the support of $h_i^k$ by definition.
We treat each integral separately. In the second one, we use the bounds \eqref{d est2} and \eqref{d est3}, as well as \eqref{est momentum} and \eqref{dxi} and the definition of $S$ to see that
 \begin{align}
\mel\left|\int_\Omega \left((\mathrm{D}_{X_i}+\mathrm{D}_x)h_i^k +\sum_{j\neq i} \mathrm{D}_{X_j}h_i^k \cdot\de_t X_j \right)\omega_i^f \dx\right|\nonumber\\
&\lesssim \int_\Omega \left(kd_i(x)^{k-1}+d_i^k|\nabla\eta_\eps(d_i(x))|\right)\left(|(\mathrm{D}_{X_i}+\mathrm{D}_x)d_i|| \de_t X_i| +\sum_{j\neq i} |\mathrm{D}_{X_j}d_i||\de_t X_j| \right)|\omega_i^f(x)|\dx\nonumber\\
&\lesssim  k\int_\Omega d_i(x)^{k+2}|\omega_i^f(x)|\dx \lesssim kS^2M_k+\mathcal{D}C^k\eps^{k+2},\label{M est1}
\end{align}
uniformly in $k$, where we also made use of the fact that $|\nabla\eta_\eps|\approx \eps^{-1}\approx d_i(x)^{-1}$ on the set where $\nabla\eta_\eps$ does not vanish.

In the other summand, we make use of the orthogonality identity \eqref{d ort} and the fact that $h_i^k$ is a function of $d_i$ to see that \begin{align*}
\mel\int_\Omega \left(-\mathrm{D}_x h_i^k\cdot \de_t X_i+\mathrm{D}_x h_i^k\cdot u(x)\right)\omega_i^f(x)\dx\\
&=\int_\Omega \bigl(-\de_t X_i+u_i^p(X_i)+u(x)-u^p(x)\bigr)\cdot\mathrm{D}_x h_i^k(x)\omega_i^f(x)\dx.
\end{align*} 
Using the bound on $|\de_t X_i-u_i^p(X_i)|$ from Proposition \ref{P5}, as well as the bounds \eqref{bd deri k} and \eqref{est momentum}, we may estimate the first difference as \begin{align}
\mel\int_\Omega \left|\de_t X_i-u_i^p(X_i)\right|\left|\mathrm{D}_x h_i^k(x)\right||\omega_i^f(x)|\dx\lesssim \left(\eps^2\mathcal{D}^\frac{1}{2}+M_2\right)\int_\Omega \left|\mathrm{D}_xh_i^k(x)\right||\omega_i^f(x)|\dx\nonumber\\
&\lesssim \left(\eps^2\mathcal{D}^\frac{1}{2}+M_2\right)\left(C^k\eps^{k-1}\mathcal{D}+kM_{k-1}\right).\label{M est2}
\end{align}
To estimate the second difference on the other hand, we use the Biot-Savart law and split the velocity into the contributions of $\omega_i^c+\sum_{j\neq i} \omega_j$ and $\omega_i^f$ to the effect that \begin{align}
\mel\left|\int_\Omega (u(x)-u^p(x))\cdot\mathrm{D}_x h_i^k(x)\omega_i^f\dx\right| \nonumber\\
&\leq \left|\int_\Omega\left( \int_\Omega\nabla^\perp G(x,y) \left(\omega_i^c(y)+\sum_{j\neq i} \omega_j(y)\right)\dy-u^p(x)\right)\cdot\mathrm{D}_x h_i^k(x)\omega_i^f(x)\dx\right|\nonumber\\
&\quad+\left|\int_{\Omega^2} \left(\frac{(x-y)^\perp}{2\pi|x-y|^2}-\nabla^\perp\gamma(x,y)\right)\cdot \mathrm{D}_xh_i^k(x)\omega_i^f(x)\omega_i^f(y)\dx\dy\right|.\label{ff part}
\end{align}
We estimate both summands separately.
In the first one, we make use of Corollary \ref{cor u bd 1} and observe that we may estimate it by \begin{align}
\mel\left|\int_\Omega\left( \int_\Omega\nabla^\perp G(x,y) \left(\omega_i^c(y)+\sum_{j\neq i} \omega_j(y)\right)\dy-u^p(x)\right)\cdot\mathrm{D}_x h_i^k(x)\omega_i^f(x)\dx\right|\nonumber\\
&\lesssim \int_\Omega (\mathcal{D}^\frac{1}{2} \eps^2|x-X_i|^{-3}+M_1|x-X_i|^{-2}+\mathcal{D}|x-X_i|^{-1})|\mathrm{D}_x h_i^k(x)||\omega_i^f(x)|\dx\nonumber\\
&\lesssim k\int_\Omega \left(\eps^2\mathcal{D}^\frac{1}{2}d_i^{k-4}+M_1d_i^{k-3}+\mathcal{D}d_i^{k-2}\right)\omega_i^f(x)\dx+C^k\left(\eps^{k-2}\mathcal{D}^\frac{1}{2}+M_1\eps^{k-3}+\eps^{k-2}\mathcal{D}\right)\mathcal{D}\nonumber\\
&\lesssim k\left(\eps^2\mathcal{D}^\frac{1}{2}M_{k-4}+M_1M_{k-3}+\mathcal{D}M_{k-2}\right)+C^k\left(\eps^{k-2}\mathcal{D}^\frac{1}{2}+M_1\eps^{k-3}\right)\mathcal{D},\label{M est3}
\end{align}
where we have used that $|x-X_i|\geq \eps$ on $\supp\omega_i^f$ and used the estimate \eqref{bd deri k} on the derivative of $h_i^k$ as well as \eqref{est momentum} and \eqref{est far part} and dropped some redundant terms in the last step.

We may also absorb the term $M_1M_{k-3}$ into the others here, because by Hölder and the estimates \eqref{est far part} and \eqref{est momentum} it holds that \begin{align}
\mel M_1M_{k-3}\leq \left|\int_\Omega \omega_i^f\dx\int_\Omega d_i(x)^{k-2}\omega_i^f\dx\right|\lesssim \mathcal{D}M_{k-2}+C^k\eps^{k-2}\mathcal{D}^2.\label{M est4}
\end{align}

\noindent In the second integral in \eqref{ff part} on the other hand, we may use the antisymmetry of the full-space Biot-Savart law and the smoothness of $\gamma$ to see that \begin{align}
\mel\left|\int_{\Omega^2}  \left(\frac{(x-y)^\perp}{2\pi|x-y|^2}-\nabla^\perp\gamma(x,y)\right)\cdot\mathrm{D}_x h_i^k\omega_i^f(x)\omega_i^f(y)\dx\dy\right|\nonumber\\
&\lesssim \left|\int_{\Omega^2} \frac{(\mathrm{D}_x h_i^k(x)-\mathrm{D}_x h_i^k(y))\cdot(x-y)^\perp}{|x-y|^2}\omega_i^f(x)\omega_i^f(y)\dx\dy\right|\nonumber\\
&\quad+\left|\int_{\Omega^2} |\mathrm{D}_x h_i^k(x)|\omega_i^f(x)\omega_i^f(y)\dx\dy\right|.\label{bs split}
\end{align}
Regarding the second integral, we see from \eqref{bd deri k} and \eqref{est momentum} that \begin{align}
\left|\int_{\Omega^2} |\mathrm{D}_x h_i^k(x)|\omega_i^f(x)\omega_i^f(y)\dx\dy\right|\lesssim C^k\eps^{k-1} \mathcal{D}^2+k\mathcal{D}M_{k-1}.\label{M est5}
\end{align}


\noindent In the first integral in \eqref{bs split}, we would like to gain an additional logarithm by using \eqref{est log},  therefore we split $\Omega^2$ in the different regions \begin{align*}
&\Omega_{1}:=\left\{(x,y)\in \Omega^2\,\Big|\,|x-X_i|\leq \eps^{1-\frac{1}{100}},\,|y-X_i|\leq \eps^{1-\frac{2}{100}}\right\}\\
&\Omega_{2}:=\left\{(x,y)\in \Omega^2\,\Big|\,|x-X_i|\geq \eps^{1-\frac{1}{100}}\right\}\\
&\Omega_{3}:=\left\{(x,y)\in \Omega^2\,\Big|\,|x-X_i|\leq \eps^{1-\frac{1}{100}},\,|y-X_i|\geq \eps^{1-\frac{2}{100}}\right\}.
\end{align*}
By the symmetry in $x$ and $y$, we observe that \begin{align*}
 \mel\left|\int_{\Omega^2} \frac{(\mathrm{D}_x h_i^k(x)-\mathrm{D}_x h_i^k(y))\cdot(x-y)^\perp}{|x-y|^2}\omega_i^f(x)\omega_i^f(y)\dx\dy\right|\leq\\
 &2\left|\int_{\Omega_1}\mathds{1}_{|y-X_i|>|x-X_i|}\frac{(\mathrm{D}_x h_i^k(x)-\mathrm{D}_x h_i^k(y))\cdot(x-y)^\perp}{|x-y|^2}\omega_i^f(x)\omega_i^f(y)\dx\dy\right|\\
 &+2\left|\int_{\Omega_2}\mathds{1}_{|y-X_i|>|x-X_i|}\frac{(\mathrm{D}_x h_i^k(x)-\mathrm{D}_x h_i^k(y))\cdot(x-y)^\perp}{|x-y|^2}\omega_i^f(x)\omega_i^f(y)\dx\dy\right|\\
 &+2\left|\int_{\Omega_3}\mathds{1}_{|y-X_i|>|x-X_i|}\frac{(\mathrm{D}_x h_i^k(x)-\mathrm{D}_x h_i^k(y))\cdot(x-y)^\perp}{|x-y|^2}\omega_i^f(x)\omega_i^f(y)\dx\dy\right|.
\end{align*}
We use the triangle inequality and estimate each integral separately.
In the integral over $\Omega_1$ we may use the estimates \eqref{bd 2deri h} and \eqref{est momentum} to see that \begin{align}
    \mel\left|\int_{\Omega_1}\mathds{1}_{|y-X_i|>|x-X_i|}\frac{(\mathrm{D}_x h_i^k(x)-\mathrm{D}_x h_i^k(y))\cdot(x-y)^\perp}{|x-y|^2}\omega_i^f(x)\omega_i^f(y)\dx\dy\right|\nonumber\\
    &\lesssim \biggl|\int_{\Omega_1}\mathds{1}_{|y-X_i|>|x-X_i|}\Big(k^2(h_i^{k-2}(x)+h_i^{k-2}(y))\nonumber\\
   & \mkern46mu+C^k\eps^{k-2}\mathds{1}_{\{\min(|x-X_i|,\,|y-X_i|)\lesssim 100N_1\eps\}}\Big)\omega_i^f(x)\omega_i^f(y)\dx\dy\biggr|\nonumber\\
    &\lesssim k^2\eps^{2-\frac{4}{100}}\mathcal{D}M_{k-4}+C^k\eps^{k-2}\mathcal{D}^2.\label{M est6}
\end{align}
Similarly, in the integral over $\Omega_2$ we may also employ \eqref{bd 2deri h} as well as \eqref{est far part}, to see that \begin{align}
\mel\left|\int_{\Omega_2}\mathds{1}_{|y-X_i|>|x-X_i|}\frac{(\mathrm{D}_x h_i^k(x)-\mathrm{D}_x h_i^k(y))\cdot(x-y)^\perp}{|x-y|^2}\omega_i^f(x)\omega_i^f(y)\dx\dy\right|\nonumber\\
&\lesssim \left|\int_{\{x:|x-X_i|,\,|y-X_i|\geq \eps^{1-\frac{1}{100}}\}}k^2\left(h_i^{k-2}(x)+h_i^{k-2}(y)\right)\omega_i^f(x)\omega_i^f(y)\dx\dy\right| \nonumber\\
&\lesssim k^2M_{k-2}\mathcal{D}|\log\eps|^{-1}.\label{M est7}
\end{align}
In the integral over $\Omega_3$ we expand the product in the numerator and see from \eqref{ort deri est} that for $(x,y)\in \Omega_3$ we have \begin{align}
&\left|\frac{(\mathrm{D}_x h_i^k(x)-\mathrm{D}_x h_i^k(y))\cdot(x-y)^\perp}{|x-y|^2}\right|\lesssim \frac{|\mathrm{D}_x h_i^k(y)||x-X_i|+|\mathrm{D}_x h_i^k(x)||y-X_i|}{|x-y|^2}\nonumber\\
&\quad+k\frac{|\scalar{\mathrm{D}_x d_i(x)}{(x-X_i)^\perp}|\left(h_i^{k-1}(x)+|\nabla\eta_\eps(d_i(x))|d_i(x)^k\right)}{|x-y|^2}\nonumber\\
&\quad+k\frac{|\scalar{\mathrm{D}_y d_i(y)}{(y-X_i)^\perp}|\left(h_i^{k-1}(y)+|\nabla\eta_\eps(d_i(y))|d_i(y)^k\right)}{|x-y|^2}\nonumber\\
&\lesssim k\left(h_i^{k-3}(y)+h_i^{k-3}(x)+C^k\eps^{k-3}\right)\frac{|x-X_i||y-X_i|^2}{|x-y|^2}+k\frac{h_i^{k+2}(y)+h_i^{k+2}(x)+C^k\eps^{k+2}}{|y-X_i|^2}\nonumber\\
&\lesssim k\left(h_i^{k-3}(y)+h_i^{k-3}(x)+C^k\eps^{k-3}\right)|x-X_i|+k\frac{h_i^{k+2}(y)+h_i^{k+2}(x)+C^k\eps^{k+2}}{|y-X_i|^2},\label{est frac}
\end{align}
where we further used that $|x-X_i|<|y-X_i|\approx |y-x|$ by the definition of $\Omega_3$ and used the equivalence of $d_i$ with the distance to $X_i$ (Lemma \ref{d is nice}) a couple of times.

We can further estimate for $(x,y)\in \Omega_3$, using Lemma \ref{d is nice} and the lower bound for $|y-X_i|$ \begin{align*}
k\frac{h_i(y)^{k+2}+h_i^{k+2}(x)+C^k\eps^{k+2}}{|y-X_i|^2}\lesssim k\left(h_i^{k}(x)+h_i^{k}(y)+C^k\eps^{k}\right).
\end{align*}
Each of these summands, except $h_i^{k}(y)$, is much smaller than the other group of terms in \eqref{est frac}.
Hence, using the definition of $\Omega_3$, we have that  
 \begin{align}
\mel\left|\int_{\Omega_3}\mathds{1}_{|y-X_i|>|x-X_i|}\frac{(\mathrm{D}_x h_i^k(x)-\mathrm{D}_x h_i^k(y))\cdot(x-y)^\perp}{|x-y|^2}\omega_i^f(x)\omega_i^f(y)\dx\dy\right|\nonumber\\
&\lesssim \left|\int_{\Omega_3} k\left(\left(h_i^{k-3}(y)+h_i^{k-3}(x)+C^k\eps^{k-3}\right)|x-X_i|+h_i^{k}(y)\right)\omega_i^f(x)\omega_i^f(y)\dx\dy\right|\nonumber\\
&\lesssim k\eps^{1-\frac{1}{100}}M_{k-3}\mathcal{D}+C^k\eps^{k-2-\frac{1}{100}}\mathcal{D}^2+kM_{k}\mathcal{D},\label{M est8}    
\end{align}
where we also used \eqref{est far part} again.

Finally to obtain \eqref{d est m}, we put \eqref{M est1}, \eqref{M est2}, \eqref{M est3}, \eqref{M est4}, \eqref{M est5}, \eqref{M est6}, \eqref{M est7} and \eqref{M est8} together drop some redundant terms using that by the assumption \eqref{control cmass} the sequence $M_l$ is decreasing and it holds that $M_l\lesssim \mathcal{D}$ by \eqref{est far part}.



\subsection{Proof of the differential estimate \texorpdfstring{\eqref{d est s}}{(3.19)} on \texorpdfstring{$S$}{S}}
\label{sec proof S}

We compute the derivative of $S$, by definition it is \begin{align*}
\mel S'\leq \max_{i=1,\dots, n}\sup_{x:\,d_i(x)=S}|u(x)\cdot \mathrm{D}_x d_i(x)+\de_t d_i(x)|\\
&=\max_{i=1,\dots, n}\sup_{x:\,d_i(x)=S}\left|u(x)\cdot \mathrm{D}_x d_i(x)+\sum_{j=1}^n\mathrm{D}_{X_j} d_i(x)\cdot \de_t X_j\right|.
\end{align*}
Using Lemma \ref{d is nice} and the bound \eqref{dxi} on the $\de_t X_j$ we see that \begin{align*}
    S'\lesssim \sup_{i=1,\dots, n ;\:\:x:\,d_i(x)=S}\Bigl(|(u(x)-\de_t X_i)\cdot \mathrm{D}_x d_i(x)|+|x-X_i|^3\Bigr).
\end{align*}
Using \eqref{d est1}, we see that $\sup |x-X_i|^3\lesssim S^3$.
In the other summand, we use \eqref{d ort}, Proposition \ref{P5}, and Corollary \ref{cor u bd 1} to see that \begin{align*}
\mel \sup_{i=1,\dots, n;\:\: x:\,d_i(x)=S}|(u(x)-\de_t X_i)\cdot \mathrm{D}_x d_i(x)|\\
&= \sup_{i=1,\dots, n;\:\: x:\,d_i(x)=S}\left|\left(u(x)-\de_t X_i-u^p(x)+u_i^p(X_i)\right)\cdot \mathrm{D}_x d_i(x)\right|\\
&\lesssim \sup_{i=1,\dots, n;\:\: x:\,d_i(x)=S}\left|\int_\Omega\nabla^\perp G(x,y) \omega_i^f(y)\dy\right|+\eps^2\mathcal{D}^\frac{1}{2}+M_2\\
&\quad\quad+\mathcal{D}^\frac{1}{2}\eps^2|x-X_i|^{-3}+\mathcal{D}|x-X_i|^{-1}+M_1|x-X_i|^{-2}\\
&\lesssim \sup_{i=1,\dots, n;\:\: x:\,d_i(x)=S}\left|\int_\Omega\nabla^\perp G(x,y) \omega_i^f(y)\dy\right|+\eps^2\mathcal{D}^\frac{1}{2}+M_2+\mathcal{D}^\frac{1}{2}\eps^2S^{-3}\\
&\quad\quad+\mathcal{D}S^{-1}+M_1S^{-2}\\
&\lesssim \sup_{i=1,\dots, n;\:\: x:\,d_i(x)=S}\left|\int_\Omega\nabla^\perp G(x,y) \omega_i^f(y)\dy\right|+\mathcal{D}^\frac{1}{2}\eps^2S^{-3}+\mathcal{D}S^{-1}.
\end{align*}
where we used the trivial inequality $M_l\lesssim \mathcal{D}S^l$, which follows directly from the definitions as well as \eqref{est far part}, and $S\lesssim 1$ to absorb redundant terms in the last step.

In order to estimate the integral, we split it into the parts where $d_i(y)\leq \frac{1}{2}S$ and where $d_i(y)\geq \frac{1}{2}S$. On the former, it holds that $|X_i-y|\leq \frac{2}{3}S$ and $|X_i-x|\geq \frac{3}{4}S$ for $x$ with $d_i(x)=S$ by \eqref{d est1} and \eqref{control cmass}. Therefore, for such $y$ and $x$ it holds that $|x-y|\gtrsim S$ and 
 we can  estimate the integral, using the definition of $G$ and the boundedness of $\gamma$, together with \eqref{est far part} by \begin{align*}
\sup_{i=1,\dots, n;\:\: x:\,d_i(x)=S}\left|\int_{\{y\,|\,d_i(y)\leq \frac{1}{2}S\}}\nabla^\perp G(x,y) \omega_i^f(y)\dy\right|\lesssim \mathcal{D}S^{-1}.
\end{align*}
For the other contribution, we use that by e.g.\ the bathtub principle, \cite[Thm. 1.14]{LiebLoss} it holds for any function $f\in L^1\cap L^\infty$ that \begin{align*}
\left|\int_{\R^2}\frac{1}{|x-y|}f(y)\dd y\right|\lesssim \norm{f}_{L^1}^\frac{1}{2}\norm{f}_{L^\infty}^\frac{1}{2}
\end{align*}
and hence, using \eqref{ref green} and that $|\nabla\gamma|$ is bounded, we have that \begin{align*}
\mel\sup_{i=1,\dots, n;\:\: x:\,d_i(x)=S}\left|\int_{\{y\,|\,d_i(y)\geq \frac{1}{2}S\}}\nabla^\perp G(x,y) \omega_i^f(y)\dy\right|\\
&\lesssim \sup_{i=1,\dots, n;\:\: x:\,d_i(x)=S}\norm{\mathds{1}_{\{y\,|\,d_i(y)\geq \frac{1}{2}S\}}\omega_i^f}_{L^1}+\norm{\mathds{1}_{\{y\,|\,d_i(y)\geq \frac{1}{2}S\}}\omega_i^f}_{L^1}^\frac{1}{2}\norm{\mathds{1}_{\{y\,|\,d_i(y)\geq \frac{1}{2}S\}}\omega_i^f}_{L^\infty}^\frac{1}{2}\\
&\lesssim \eps^{-1}\norm{\mathds{1}_{\{y\,|\,d_i(y)\geq \frac{1}{2}S\}}\omega_i^f}_{L^1}^\frac{1}{2},
\end{align*}
where we used \eqref{A4} in the last step.
We may further estimate this $L^1$-norm with the higher order momenta by \begin{align*}
\eps^{-1}\norm{\mathds{1}_{\{y\,|\,d_i(y)\geq \frac{1}{2}S\}}\omega_i^f}_{L^1}^\frac{1}{2}\lesssim \eps^{-1}C^{\frac{k}{2}}S^{-\frac{k}{2}}M_k^\frac{1}{2}.
\end{align*}

%
%
%
%
%
%
%
\noindent Combining the previous estimates shows \eqref{d est s}.

\hfill\qedsymbol

\subsection{Proof of Proposition \ref{eff est}}\label{Gronwallsec}
We only consider $k\lesssim |\log\eps|$, otherwise there is nothing to show. 

We will use the following approach: We define a "stopping time" $T$ as follows: \begin{align}
\tst:=\inf \left\{t\geq 0\,\bigg|\, (1+t)|\log\eps|(\sup_{s\leq t}S(s)^2)+t\eps+\frac{M_1(t)}{\eps}+\mathcal{D}(t)\geq c_0 \right\},\label{def T}
\end{align}
where $c_0<<1$ is some sufficiently small positive number which may depend on $b$, $N_1$, $N_2$, etc.\ but not on $\eps$. Clearly $\tst>0$ for sufficiently small $\eps$ because $M_1(0)=0$ and $\mathcal{D}(0)\lesssim \eps^\beta$.
In particular, the Assumption \eqref{control cmass} holds at least up to the time $\tst$ if $c_0$ is chosen sufficiently small in relation to $C_1$.
\subsubsection{The estimate \eqref{bd d} for $\mathcal{D}$}
Using the equivalence between $\tilde{\mathcal{D}}$ and $\mathcal{D}(t)$ from Lemma \ref{lem d} and the fact that $M_k\lesssim S^k\mathcal{D}$ by the definition and \eqref{est far part}, we have \begin{align}
\tilde{\mathcal{D}}'(t)\lesssim \eps^2\tilde{\mathcal{D}}^\frac{1}{2}+\eps^4+S^2\tilde{\mathcal{D}}\quad \text{ for $t\leq \tst$}.\label{dtilde}
\end{align}
By Gronwall and the fact that $\tilde{\mathcal{D}}(0)\lesssim \eps^\beta+\eps^4$ (by the assumption \eqref{A6} and \eqref{d eq2}), we obtain that if $c_0$ is sufficiently small, then \begin{align*}
\tilde{\mathcal{D}}(t)\lesssim \exp\left(C\int_0^tS^2(s)\dd s\right)\left(\eps^4(1+t)^2+\eps^{\beta}\right),
\end{align*}
where $C$ is the implicit constant in \eqref{dtilde}, which in particular does not depend on $c_0$ or $\eps$. We hence obtain that \begin{align*}
\mathcal{D}(t)\lesssim \eps^4(1+t)^2+\eps^\beta \quad \text{ for $t\leq \tst$},
\end{align*}
as desired.
\subsubsection{The inequality for $M_k$.}

Using \eqref{bd d} and the definition of $\tst$, we  can first simplify the right-hand side of \eqref{d est m} to \begin{align*}
M_k'(t)&\lesssim kS(t)^2M_k(t)+ C^k\eps^k\left(\eps^{\frac{3\beta}{2}-2}+\eps^4(1+t)^3+\eps^{2\beta-3}\right)+k\left(\eps^{2+\frac{\beta}{2}}+\eps^4(1+t)\right)M_{k-4}\\
&\quad+k(\eps^\beta+\eps^4(1+t)^2)M_{k-2}+k\left(\eps^{\beta+1-\frac{1}{100}}+\eps^{5-\frac{1}{100}}(1+t)^2\right)M_{k-3},
\end{align*}
for $t\leq T$, where we further used that $k\lesssim |\log\eps|$ and $t\eps\lesssim 1$ by assumption and estimated $M_1\lesssim \mathcal{D}$ by \eqref{est far part}.

If we further use that by Hölder, \eqref{est far part} and \eqref{est momentum} we have for $0\leq l\leq k$\begin{align*}
M_{k-l}\lesssim C^{k-l}\eps^{k-l}\mathcal{D}+\mathcal{D}^\frac{l}{k}M_k^\frac{k-l}{k},
\end{align*}
then we obtain that \begin{align*}
M_k'(t)&\lesssim kS(t)^2M_k(t)+ C^k\eps^k\left(\eps^{\frac{3\beta}{2}-2}+\eps^4(1+t)^3+\eps^{2\beta-3}\right)\\
&\quad+k\left(\eps^{2+\frac{\beta}{2}}+\eps^4(1+t)\right)\left(\eps^\beta+\eps^4(1+t)^2\right)^\frac{4}{k}M_{k}^\frac{k-4}{k}+k\left(\eps^\beta+\eps^4(1+t)^2\right)^{1+\frac{2}{k}}M_{k}^\frac{k-2}{k},
\end{align*}
for $t\leq \tst$, here we have also used that the term coming from $M_{k-3}$ can be estimated from above by the terms containing $M_k^{\frac{k-4}{k}}$ and $M_k^{\frac{k-2}{k}}$ by Young's inequality.

Gronwall and the fact that initially $M_k(0)=0$ by definition then yield the upper bound of \begin{align*}
M_k(t)&\lesssim C^k\exp\left(k\int_0^t S(s)^2\dd s\right)\bigg(\eps^k\left(\left(\eps^{\frac{3\beta}{2}-2}+\eps^{2\beta-3}\right)(1+t)+\eps^4(1+t)^4\right)\\
&\quad+(1+t)^\frac{k}{4}\left(\eps^{2+\frac{\beta}{2}}+\eps^4(1+t)\right)^\frac{k}{4}\left(\eps^\beta+\eps^4(1+t)^2\right)+(1+t)^\frac{k}{2}\left(\eps^\beta+\eps^4(1+t)^2\right)^{\frac{k}{2}+1}\bigg)\,,
\end{align*}
for $t\leq \tst$ if $c_0$ is small enough so that the exponential term is of order $1$. Using that $t\eps\leq c_0$ is small if $t\leq \tst$ and distinguishing the cases $\beta\geq 2$ and $\beta\leq 2$, this can be simplified to \begin{align*}M_k&\lesssim  \begin{cases} C^k\left(\left(\eps^k(1+t)^\frac{k}{2}+\eps^{k(\frac{1}{2}+\frac{\beta}{8})}(1+t)^\frac{k}{4}\right)\left(\eps^4(1+t)^2+\eps^\beta\right)+\eps^{k-2+\frac{3}{2}\beta}(1+t)\right)&\text{ for $\beta\geq 2$}\\
C^k\left(\eps^{\beta(\frac{k}{2}+1)}(1+t)^\frac{k}{2}+\eps^{k(\frac{1}{2}+\frac{\beta}{8})+\beta}(1+t)^\frac{k}{4}+\eps^{k-3+2\beta}(1+t)\right)&\text{ for $\beta\leq 2$}\end{cases}
\end{align*}
which is precisely \eqref{bd m}.

\subsubsection{Resolving the estimate for $S$}
Using the previous estimates \eqref{bd d} and \eqref{bd m}, the estimate \eqref{d est s} turns into \begin{align*}
S'\lesssim \left(\eps^4(1+t)+\eps^{2+\frac{\beta}{2}}\right)S^{-3}+S^3+\left(\eps^\beta+\eps^4(1+t)^2\right)S^{-1}+C^k\eps^{-1}S^{-\frac{k}{2}}M_k^\frac{1}{2}\quad \text{ for $t\leq \tst.$}
\end{align*}
Gronwall and the fact that $S(0)\approx \eps$ by the assumption \eqref{A2} now give the estimate \begin{align*}
S\lesssim \eps^{\min(1,\frac{\beta}{2})}(1+t)^\frac{1}{2}+\eps^{\frac{1}{2}+\frac{\beta}{8}}(1+t)^\frac{1}{4}+\eps^{-\frac{2}{k+2}}(1+t)^\frac{2}{k+2}M_k^\frac{1}{k+2}\quad \text{ for $t\leq \tst.$}
\end{align*}
If we now take $k\approx |\log\eps|$, then it holds that $\eps^\frac{1}{k}\approx (1+t)^\frac{1}{k}\approx 1$ for $t\leq \tst$ (since $\tst\lesssim \eps^{-1}$ by definition). Combining this with the previous estimate for $S$ then yields \begin{align*}
S(t)\lesssim \eps^{\min(1,\frac{\beta}{2})}(1+t)^\frac{1}{2}+\eps^{\frac{1}{2}+\frac{\beta}{8}}(1+t)^\frac{1}{4}\quad \text{ for $t\leq \tst$.}
\end{align*}
It remains to check that the time $\tst$ indeed has the lower bound in \eqref{bd t}. This is immediate from the bounds for every part of the condition \eqref{def T} defining $\tst$, except the one containing $M_1$. For this part, we use that $M_1\lesssim S\mathcal{D}$ by the definition and \eqref{est far part}.

Using the bounds \eqref{bd d} and \eqref{bd s}, this gives a bound of \begin{align*}
M_1\lesssim S\mathcal{D}\lesssim  \eps+\eps^\beta\left(\eps^{\frac{\beta}{2}}(1+t)^\frac{1}{2}+\eps^{\frac{1}{2}+\frac{\beta}{8}}(1+t)^\frac{1}{4}\right)
\end{align*}
for $t\eps\lesssim 1$, which yields the limitations for $\beta<\frac{4}{5}$, while the restrictions on $\tst$ for $\beta\geq \frac{4}{5}$ come from the requirement that $T|\log\eps|S(T)^2\leq c_0$. 

\subsection{Proof of Theorem \ref{T1'}}
\label{sec38}
The idea is to show that the usual argument that the conservation of momentum and energy (of the point vortex system) prevent collisions still works here, since these quantities are almost preserved, as the following Lemma shows. 
\begin{lemma}
Assume that the assumptions \eqref{A1}-\eqref{A7} on the initial data hold with some fixed $b$, then up to the time $T=T(b)$ from Theorem \ref{T1} we have the following estimates for every $t\in [0,T)$\begin{align}
&\left|\sum_{i=1}^n a_i \left(|X_i^0|^2-|X_i(t)|^2\right)\right|\lesssim \eps \left(\sum_{i=1}^n |X_i(t)|+|X_i^0|\right)\label{almost mom}\\
&\left|\sum_{i\neq j} a_ia_j \left(\log|X_i(t)-X_j(t)|-\log|X_i^0-X_j^0|\right)\right|\lesssim \eps^{\beta}+\eps^4(1+t)^2,\label{almost energy}
\end{align}
(with an implicit constant depending on $b$) \textbf{or} the assumption \eqref{A7} is violated at some time before $t$.
\end{lemma}
\begin{proof}
For \eqref{almost mom}, we first note that the total angular momentum of the fluid, $\int_{\R^2} \omega(x)|x|^2\dx$, is a conserved quantity. On the other hand, as long as the estimates in Proposition \eqref{eff est} hold, we can also estimate \begin{align*}
\mel\left| \int_{\R^2} \omega(x)|x|^2\dx-\sum_{i=1}^n a_i |X_i|^2\right|\leq\int_{\R^2} \sum_{i=1}^n|\omega_i| \left||X_i|^2-|x|^2\right|\dx\\
&\lesssim \sum_{i=1}^n\int_{\R^2}|\omega_i| \left(|X_i||X-x_i|+|X_i-x|^2\right)\dx\\
&\lesssim (\eps+M_1)\sum_{i=1}^n  |X_i|+\eps^2+M_2\lesssim \eps\left(1+\sum_{i=1}^n  |X_i|\right).
\end{align*}
Here, the penultimate step follows from splitting $\omega_i=\omega_i^c+\omega_i^f$, using \eqref{est momentum} for the far part and noting that $|X_i-x|\lesssim \eps$ on the support of the close part; the last step follows from \eqref{control cmass}, which holds by Proposition \ref{eff est}.
Applying this estimate at both the times $0$ and $t$ and using the triangle inequality shows \eqref{almost mom}.\smallskip

The estimate \eqref{almost energy} uses a similar idea; the energy $\sum_{i\neq j}a_ia_j\log|X_i-X_j|$ is $\tilde{\mathcal{D}}$, up to terms which are conserved under the evolution and a constant factor as one can directly see from the definition \eqref{def tilde D}. Hence it holds that \begin{align*}
\left|\sum_{i\neq j} a_ia_j \left(\log|X_i(t)-X_j(t)|-\log|X_i^0-X_j^0|\right)\right|\lesssim |\tilde{\mathcal{D}}(0)-\tilde{\mathcal{D}}(t)|.
\end{align*}
Lemma \ref{lem d} and \eqref{bd d} yield the claim.

\end{proof}

We now pick $b_0<\frac{1}{10}$, depending on the $X_i^0$, so that on the one hand \eqref{A7} holds for the initial data and $b=2b_0$ and on the other hand so that \begin{align}
\left(\min_i a_i\right)^2|\log(2b_0)|-\left(\frac{n(n-1)}{2}-1\right)\left(\max_{i}a_i\right)^2\left|\log L\right|>1+2\left|\sum_{i\neq j}a_ia_j\log|X_i^0-X_j^0|\right|,\label{def b0}
\end{align}
where \begin{align*}
    L=2\sqrt{\left(\min_i a_i\right)^{-1}\left(1+2\sum_{i=1}^{n} a_i |X_i^0|^2\right)}.
\end{align*}
This is possible because the left-hand side term goes to $+\infty$ as $b_0\searrow 0$ and the right-hand side does not depend on $b_0$.

To prove the theorem, it then suffices to show that for sufficiently small $\eps$ (the smallness depending on $b_0$), the assumption \eqref{A7} (with $b=b_0$) can not be violated as long as the estimates \eqref{almost mom} and \eqref{almost energy} hold.

We first note that for sufficiently small $\eps$ it holds that \begin{align}
\sum_{i=1}^n a_i|X_i(t)|^2< 1+2\sum_{i=1}^n a_i|X_i^0|^2\label{mom est}
\end{align}
(up to the time $T=T(b_0)$ from Thm.\ \ref{T1}) if \eqref{A7} is not violated before $t$.

Indeed, this follows from \eqref{almost mom} and $|X_i(t)|\leq 1+|X_i(t)|^2$ by picking $\eps$ small enough so that the quadratic terms can be reabsorbed into the left-hand side.

In particular, \eqref{mom est} implies that \begin{align}
    |X_i(t)|<\frac{L}{2}.\label{est L}
\end{align}
Similarly, we see that we have the same estimate for the energy \begin{align}
\left|\sum_{i\neq j} a_ia_j\log|X_i(t)-X_j(t)|\right|< 1+2\left|\sum_{i\neq j} a_ia_j\log|X_i^0-X_j^0|\right|.\label{est energy}
\end{align}
We then compute from the definition of $b_0$ that whenever $|X_i(t)-X_j(t)|=2b_0$ we have \begin{align*}
\mel\left|\sum_{i\neq j} a_ia_j\log|X_i(t)-X_j(t)|\right|\\
&\geq\left(\min_i a_i\right)^2\left|\log(2b_0)\right|-\left(\frac{n(n-1)}{2}-1\right)\left(\max_{i}a_i\right)^2\left|\log\max_{i\neq j}|X_i-X_j|\right|.
\end{align*}
Using \eqref{est L} and the definition of $b_0$, we see that this contradicts \eqref{est energy}. Hence, the $X_i$ can never get closer than $2b_0$ to each other, and \eqref{A7} can never be violated, yielding the theorem. \hfill\qedsymbol



\textbf{Acknowledgment}
The author has received funding from the European Research Council (ERC) under the European Union’s Horizon 2020 research and innovation programme through the grant agreement 862342.

\textbf{Statement about conflicting interests}
The author declares that he has no conflict of interests.

\textbf{Data availability statement}
There is no data associated with this article.

\bibliography{Vortices}
	\bibliographystyle{acm}

\end{document}